\documentclass[12pt]{amsart}

\usepackage{diagbox}

\usepackage{geometry}
\usepackage{amsmath}
\usepackage{amssymb}
\usepackage{amsthm}
\usepackage{amsfonts, dsfont}
\usepackage{mathrsfs} 
\usepackage{paralist}
\usepackage{graphics} %% add this and next lines if pictures should be in esp format
\usepackage{epsfig} %For pictures: screened artwork should be set up with an 85 or 100 line screen
\usepackage{graphicx}  
\usepackage{epstopdf}%This is to transfer .eps figure to .pdf figure; please compile your paper using PDFLeTex or PDFTeXify.
\usepackage{epstopdf}
\usepackage{verbatim}
\usepackage{mathrsfs}
\usepackage{mathtools}
\usepackage{pstricks}
\usepackage{relsize}
\usepackage{tikz}
\usetikzlibrary{matrix}
\usepackage{subcaption}
\usepackage{pgfplots}
\usepackage{fixltx2e}
\usepackage{enumitem}
\usepackage{bm}
\usepackage{upgreek}
\usepackage{url}
\usepackage[colorlinks=true]{hyperref}
\hypersetup{urlcolor=blue, linkcolor=blue, citecolor=red}
\usepackage{cleveref}
\usepackage{bm}
\usepackage{appendix}
\usepackage{float}

\usepackage{tcolorbox} 
\usepackage{lineno}
\usepackage{algorithm2e}
\SetKwComment{Comment}{/* }{ */}

\allowdisplaybreaks

\newcommand{\D}[1]{\mbox{\rm #1}} 
\newcommand{\dd}{\D{d}}
\newcommand{\dt}{{\rm d} t}
\newcommand{\ds}{{\rm d} s}
\newcommand{\dx}{{\rm d} x}
\newcommand{\rd}{{\rm d} }
\newcommand{\RR}{{\mathbb R}}
\newcommand{\EE}{{\mathbb E}}

\newcommand{\OX}{{\overline{X}}}

\newcommand{\mc}[1]{\mathcal{#1}}

\definecolor{ao(english)}{rgb}{0.0, 0.5, 0.0}

\DeclareMathOperator*{\argmin}{argmin}

\usepackage[abbrev]{amsrefs}
\usepackage{amssymb}

\numberwithin{equation}{section}

\newtheorem{theorem}{Theorem}[section]
\newtheorem{lemma}[theorem]{Lemma}
\newtheorem{assum}[theorem]{Assumption}
\newtheorem{corollary}[theorem]{Corollary}

\newtheorem{remark}[theorem]{Remark}

\definecolor{ForestGreen}{RGB}{34,139,34}
\definecolor{ao(english)}{rgb}{0.0, 0.5, 0.0}

\begin{document}

%%% In the title, use a double backslash "\\" to show a linebreak:
%%% Use one of the following two forms:
%%% \title{Text of the title}
%%% or
%%% \title[Short form for the running head]{Text of the title}
\title[Global convergence for the rescaled CBO]
{Faithful global convergence for the rescaled Consensus--Based Optimization}
\thanks{
}

\author{Hui Huang}

\author{Hicham Kouhkouh}

\author{Lukang Sun}

\address{Hui Huang \newline
Hunan University, 
School of Mathematics, Changsha, China
}
\email{\texttt{huihuang1@hnu.edu.cn}}

\address{Hicham Kouhkouh \newline 
University of Graz, 
Department of Mathematics and Scientific Computing -- NAWI, 
Graz, Austria
}
\email{\texttt{hicham.kouhkouh@uni-graz.at}}

 \address{
Lukang Sun \newline
Technical University of Munich, 
School of Computation, Information and Technology,
Department of Mathematics, Munich, Germany
 }
 \email{\texttt{lukang.sun@tum.de}}

\date{\today}

\begin{abstract}
We analyze the Consensus-Based Optimization (CBO) algorithm with a consensus point rescaled by a  fixed parameter $\kappa \in (0,1)$. Under minimal assumptions on the objective function and the initial data, we establish its unconditional convergence to the global minimizer. Our results hold in the asymptotic regime where both the time--horizon $t \to \infty$ and the inverse--temperature $\alpha \to \infty$,  providing a rigorous theoretical foundation for the algorithm's global convergence. Furthermore, our findings extend to the case of multiple and non--discrete set of minimizers.  
\end{abstract}



\maketitle

\vspace*{-1.6em}

\section{Introduction}

Consensus--Based Optimization (CBO) is a derivative--free algorithm designed to tackle large--scale, highly non--convex optimization problems. This algorithm and its analytical foundations were established in \cite{pinnau2017consensus,carrillo2018analytical}.  It has since then been extensively studied, both theoretically and practically, demonstrating its effectiveness in solving complex optimization challenges that arise in real--world applications. 

Inspired by the collective behavior observed in social and biological systems, CBO exploits  the interactions of multiple agents, or particles, to explore the search space in a coordinated and efficient manner. The method uses a stochastic consensus mechanism, where particles are dynamically guided toward regions of lower objective values, ultimately converging to an optimal solution. This decentralized and adaptive search strategy makes CBO particularly well-suited for problems where gradient information is unavailable or unreliable.
Motivated by its broad applicability, researchers have extended and refined the original CBO framework to accommodate various settings and demonstrated strong empirical success in solving wide range of global optimization problems
\cite{fornasier2021JMLR, ha2022stochastic,borghi2022consensus,herty2025multiscale,carrillo2022consensus,carrillo2021consensus,fornasier2022anisotropic,fornasier2020consensus,fornasier2024truncated,byeon2025consensus,chen2022consensus,wei2025consensus,bonandin2025consensus,bellavia2025discrete,cipriani2022zero,huang2024fast,choi2025modified}. These advancements have further enhanced the algorithm's robustness and efficiency, making it a versatile tool across a wide range of scientific and engineering disciplines. 
Typical convergence result of CBO establishes that under suitable assumptions on the objective function (e.g., continuity, isolated global minimizers, and a “localization” condition such as the so-called “well-behaved” landscape around the global minimum), the mean-field limit of CBO approximates a Dirac mass concentrated at the global minimizer when $t$ and $\alpha$ are sufficiently large. This is a probabilistic and asymptotic guarantee in the mean-field regime, not a worst-case computational complexity statement.

One of the key contributions of this manuscript is to establish its convergence, thus providing further theoretical validation of the method.

Let us consider  the following optimization problem:
\begin{equation*}
	\text{Find }\; x_* \in \argmin_{x\in \RR^d} f(x),\,
\end{equation*}
where $f$ can be a non-convex non-smooth objective function that one wishes to minimize. Then the CBO dynamic considers the following system of $N$ interacting particles, denoted as $\{X_{\cdot}^i\}_{i=1}^N$, which satisfies
\begin{equation}\label{CBOparticle}
	\rd X_t^i=-\lambda(X_t^i - \mathfrak{m}_\alpha(\rho_t^{N})) \, \dt+\sigma D(X_t^i-\mathfrak{m}_\alpha(\rho_t^{N})) \, \rd B_t^{i},\quad i=1,\dots, N=:[N]\,,
\end{equation}
where  $\sigma>0$ is a real constant, $\rho_t^N:=\frac{1}{N}\sum_{i=1}^N\delta_{X_t^i}$ is the empirical measure associated to the particle system, and $\{B_.^i\}_{i=1}^N$ are $N$ independent $d$-dimensional Brownian motions. Moreover, the particle system is initialized with independent and identically distributed (i.i.d.) data $\{X_{0}^i\}_{i=1}^N$, where each $X_0^i$ is distributed according to a given measure $\rho_0 \in \mathscr{P}(\RR^d)$ for all $i \in [N]$.

Here, we employ the anisotropic diffusion in the sense that $D(X):=\mbox{diag}(|X_1|,\dots,|X_d|)$ for any $X\in\RR^d$, which has been proven to handle high-dimensional problems more effectively \cite{carrillo2021consensus,fornasier2022anisotropic}.
The \textit{current global consensus} point $\mathfrak{m}_{\alpha}(\rho_t^N)$ is meant to approximate the global minimum, and is defined with a weighted average of the particles as 
\begin{equation}\label{XaN}
	\mathfrak{m}_{\alpha}(\rho_t^N): = \frac{\int_{\RR^d} x \, \omega_{\alpha}^{f}(x)\; \rho_t^N(\dx)}{\int_{\RR^d}\omega_{\alpha}^{f}(x)\; \rho_t^N(\dx)}\,, \quad \text{ where } \quad \omega_\alpha^f(x):=\exp(-\alpha f(x)).
\end{equation}
This choice of weight function $\omega_\alpha^f(\cdot)$ is motivated by the well-known Laplace's principle \cite{miller2006applied,MR2571413}, see also \cite[Appendix A.2]{huang2024consensus}. Formally, this means that
\begin{equation}
    \mathfrak{m}_{\alpha}(\rho_t^N) \;\approx\; \argmin_{1 \leq i \leq N} f(X_t^i) \quad \text{as } \alpha \to \infty.
\end{equation}
In other words, as $\alpha \to \infty$, the consensus point $\mathfrak{m}_{\alpha}(\rho_t^N)$ converges to the position of the particle with the smallest objective value at time $t$, i.e., the best location found by the ensemble $\{X_t^i\}_{i=1}^N$ at time $t$. Once consensus is reached, all particles concentrate around the consensus point $\mathfrak{m}_{\alpha}(\rho_t^N)$, in the sense that
\begin{equation}
    X_{t=\infty}^i \;\approx\; \mathfrak{m}_{\alpha}(\rho_{t=\infty}^N) \quad \text{for all } i \in \{1,\dots,N\}.
\end{equation}
Consequently, provided the number of particles $N$ is sufficiently large and $\alpha \to \infty$, the ensemble concentrates near the global minimizer $x_*$. We refer to \cite[Section 1.2]{huang2025uniform} for more details.

The convergence proof of the CBO method is typically conducted in the context of large-particle limit \cite{gerber2023mean,huang2022mean}. Specifically, rather than analyzing the $N$--particle system \eqref{CBOparticle} directly, one considers the limit as $N$ approaches infinity and examines the corresponding McKean--Vlasov process $\OX_{\cdot}$, which is governed by the following equation:
\begin{equation}\label{CBO}
\begin{aligned}
    \rd \OX_t  = &  -\lambda\left(\OX_t-\mathfrak{m}_{\alpha}(\rho^{\alpha}_t)\right) \, \dt \, + \, \sigma D\left(\OX_t-\mathfrak{m}_{\alpha}(\rho^{\alpha}_t)\right) \, \rd B_t\,, \\ 
    & \text{with } \; \rho_t^\alpha \, := \, \text{Law}(\OX_t). 
\end{aligned}
\end{equation}

\subsection{The convergence problem in the literature}

 There are numerous results in the literature on the convergence analysis of CBO methods, which can be broadly categorized into two different approaches. For instance, in \cite{fornasier2024consensus}, the authors employ $\mathbb{E}[|\overline{X}_t - x_*|^2]$ as a Lyapunov functional for \eqref{CBO}. Their convergence result states that for any given accuracy $\varepsilon > 0$, there exists a time $T_* > 0$ depending on $\varepsilon$, such that for $\alpha > \alpha_0$ (also depending on $\varepsilon$), the following holds:
\begin{equation*}
    \mathbb{E}\left[|\overline{X}_t - x_*|^2\right] \leq \mathbb{E}\left[|\overline{X}_0 - x_*|^2\right] e^{-\theta t}
\end{equation*}
for some constant $\theta > 0$ and for all $t \in [0, T_*]$. Moreover, it holds that 
\[
\mathbb{E}\left[|\overline{X}_{T_*} - x_*|^2\right] \leq \varepsilon.
\]
This approach has been further extended to establish the convergence of various CBO variants, including those incorporating general constraints \cite{borghi2023constrained, beddrich2024constrained}, multi-player games \cite{chenchene2025consensus}, min-max problems \cite{borghi2024particle}, multiple-minimizer problems \cite{fornasier2025pde},  clustered federated learning \cite{carrillo2023fedcbo}, and bi-level optimization \cite{trillos2024cb}. 

However, a limitation of this method is that it does not allow for taking the limits $t \to \infty$ or $\alpha \to \infty$, meaning that convergence is guaranteed only within a neighborhood of $x_*$ with accuracy $\varepsilon$. 
More precisely, the proof of this method relies on a quantitative version of Laplace’s principle, which requires that $\rho_0(B_{r_\varepsilon}(x^*)) > 0$. However, as $\varepsilon \to 0$, the radius $r_\varepsilon \to 0$, and consequently $\rho_0(B_{r_\varepsilon}(x^*)) \to 0$, leading to a degeneracy that violates this condition.

One possible approach to achieving convergence as $t \to \infty$ is through a variance-based method described in \cite{carrillo2018analytical}. This approach involves defining the variance function
\[
V(t) = \mathbb{E}\left[|\overline{X}_t - \mathbb{E}[\overline{X}_t]|^2\right].
\]
The proof is then divided into two steps. First, it is shown that under well-prepared initial data and appropriately chosen parameters, for any fixed $\alpha$, the function $V(t)$ converges to zero exponentially fast as $t \to \infty$. This implies that $\overline{X}_t$ converges to some consensus point $x^\alpha$ as $t \to \infty$. Second, it is verified that the consensus point $x^\alpha$ can be made arbitrarily close (up to any given accuracy $\varepsilon$) to the global minimizer $x_*$ by choosing some fixed $\alpha$ that is sufficiently large.  This argument has been further extended to establish the convergence of such as particle swarm optimization \cite{huang2023global}, CBO constrained on the sphere \cite{fornasier2021JMLR}, CBO for saddle point problems \cite{huang2024consensus}, and CBO with mirror descent \cite{bungert2025mirrorcbo}. A similar idea is also employed in the convergence proof at the particle level \cite{ko2022convergence}. 
However, similar to the first method, this approach does not allow taking the limit $\alpha\to \infty $ and it only guarantees convergence within an $\varepsilon$--neighborhood of the global minimizer. More precisely, this method requires that the initial variance satisfies $V(0) \leq \frac{C}{\alpha}$ and that the initial distribution contains the global minimizer $x_*$. Consequently, as $\alpha \to \infty$, the condition $V(0) \leq \frac{C}{\alpha}$ forces the initial distribution to concentrate at $x_*$ already.

In this paper, we will establish the first \textit{true} convergence result for CBO method, demonstrating that both limits, $t \to \infty$ and $\alpha \to \infty$, can be taken. Up to our knowledge, this is the first result of this kind. The main idea relies on the following.

\subsection{Motivation and the \textit{rescaled}  CBO}

As it has been observed in \cite{huang2024self}, see also \cite[Remark 1.1]{huang2025uniform}, any Dirac measure is invariant for the dynamics \eqref{CBO} which makes it difficult to study its asymptotic properties. This non--uniqueness of an invariant probability measure is mainly due to two reasons: in the drift, the term $-X$ is not strong enough with respect to the
mean--field term $\mathfrak{m}_{\alpha}(\cdot)$, and the diffusion degenerates. In practice, in such situations, the particle population are often observed to collapse prematurely to a Dirac delta distribution before the global minimizer has been reached (see Figure \ref{fig:Rastrigin}) due to the global minimizer not being contained in the initial distribution. Moreover, on the theoretical side, while convergence in mean-field law over finite time $t<T$  and finite $\alpha$ has been achieved in the literature (as discussed in Section 1.1), proving global convergence over an infinite time horizon has remained elusive.
\begin{figure}[H]
	\centering
    \includegraphics[width=0.45\textwidth]{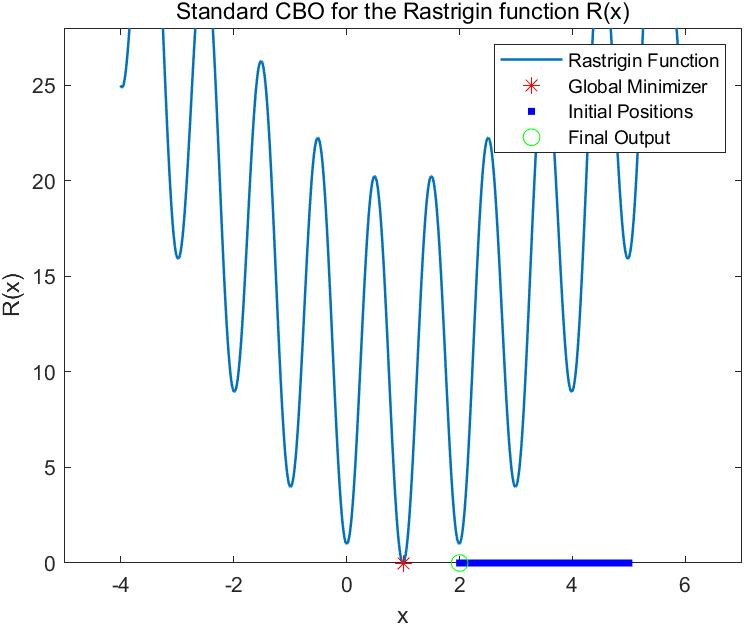} \; %\hfill
    \includegraphics[width=0.45\textwidth]{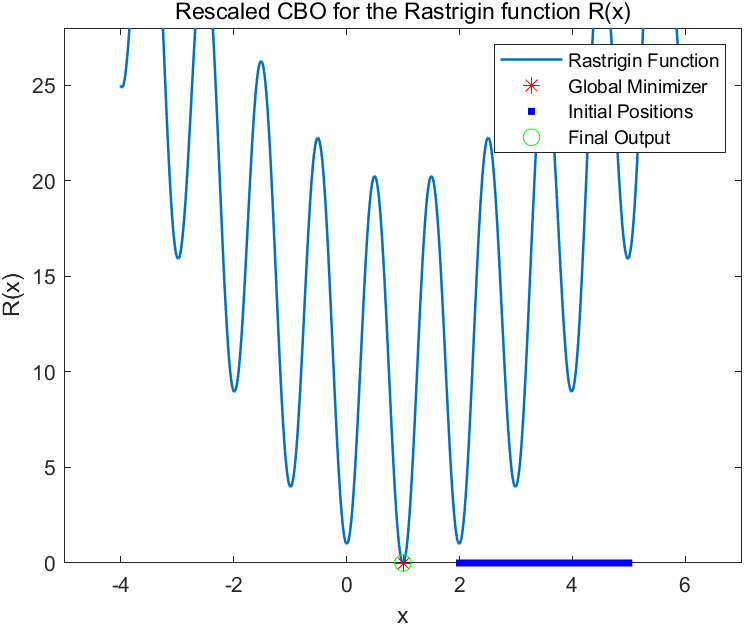}
	\caption{We apply the \textit{standard} CBO ($\kappa=1$ and $\delta=0$) and the \textit{rescaled} CBO ($\kappa=0.01$ and $\delta=5$) particle system \eqref{CBOkappa particle} to the Rastrigin function $R(x):=10+(x-1)^2-10\cos(2\pi(x-1))$, which has a unique global minimizer $x_*=1$ (the red star). The initial particles (the blue dots) are sampled uniformly in $[2,5]$ (it does not contain $x_*$). The simulation parameters are $N=100,\lambda=1,\sigma=0.5,\alpha=10^{15},\rd t=0.01$ and $T=100$. The final output is $\mathfrak{m}_{\alpha}(\rho_{T}^N)$ (the green circle). \\
    \textbf{Left:} The \textit{standard} CBO collapses prematurely to a local minimizer. \\
    \textbf{Right:} The \textit{rescaled} CBO finds the global minimizer.
    } 
	\label{fig:Rastrigin}
\end{figure} 
\noindent This is the main reason behind our modification of \eqref{CBO} whose aim is to address exactly these two mentioned issues, and which has already been proposed in \cite{huang2025uniform, huang2024self, herty2025multiscale}. 
Hence, we consider a \textit{rescaled} CBO given by
\begin{equation}\label{CBO kappa}
\begin{aligned}
    \rd \OX_t = & -\lambda\left(\OX_t - \kappa\,\mathfrak{m}_{\alpha}(\rho^{\alpha}_t)\right)\,\dt \,+ \,\sigma\left(\delta\,\mathds{I}_{d} + D\left(\OX_t - \kappa\, \mathfrak{m}_{\alpha}(\rho^{\alpha}_t)\right) \right)\,\rd B_t\\
    & \text{with } \; \rho^{\alpha}_t \, := \, \text{Law}(\OX_t). 
\end{aligned}
\end{equation}
complemented with an initial condition $\OX_{0}\sim \rho_0\in\mathscr{P}_{4}(\RR^d)$, where $\mathds{I}_{d}$ is the $d$-dimensional identity matrix, $0<\kappa<1$ is a small positive constant, and $\delta> 0$ is any positive constant. The name is justified by the fact that the contribution of the mean--field term $\mathfrak{m}_{\alpha}(\cdot)$ is re-\textit{scaled} with a parameter $\kappa$ when compared to the contribution of $-\OX$, thus promoting \textit{dissipation} ($-\OX$) over \textit{interaction} ($\mathfrak{m}_{\alpha}(\cdot)$). To implement the model numerically, one considers the  underlying system of $N$ interacting particles which is of the form
 \begin{equation}\label{CBOkappa particle}
\rd X_t^i =-\lambda \left(X_t^i - \kappa\,\mathfrak{m}_{\alpha}(\rho_t^N) \right) \dt+\sigma\left(\delta\,\mathds{I}_{d} + D(X_t^i - \kappa\, \mathfrak{m}_{\alpha}(\rho_t^N) \right)\rd B_t^i,\quad i\in [N]\,
\end{equation}
and $\rho_t^N:=\frac{1}{N}\sum_{i=1}^N\delta_{X_t^i}$ is the empirical measure associated to the particle system.

Introducing the parameters $\kappa$ and $\delta$ makes all Dirac measures not invariant any more, and intuitively, creates a disconnection between the limits in the time variable $t\to \infty$ and the inverse--temperature $\alpha \to \infty$, hence allowing to easily handle the convergence of the dynamics towards the (to--be) global minimum. We refer to \cite[Section 1.2]{huang2025uniform} for more details.

Moreover, the positive parameter $\delta\,$ guarantees ellipticity (non--degeneracy) of the PDE satisfied by the invariant probability measure of the dynamics, and grants it nice properties. More precisely, we need $\delta> 0$ in order to guarantee existence of an invariant probability measure (Theorem \ref{thm: existence} below), and to ensure suitable regularity (Theorem \ref{thm general} below). Uniqueness then follows using the $\mathbb{W}_{2}$--contraction (Theorem \ref{thm:long time} below). Its value however does not influence the asymptotic analysis that we shall perform in the sequel, or the results that we aim for. Indeed, as we will later see in \eqref{eq:to take expectation}, the constant $\delta$ appears in a martingale term which will vanish after taking the expectation, hence its value will not influence our convergence result as long as it is positive.

To further motivate the introduction of the scaling parameter $\kappa\in(0,1)$ and the non-degenerate diffusion $\delta>0$, we present some numerical experiments in Table \ref{table}.

\begin{table}[H]
\centering
\begin{tabular}{p{4cm} p{8cm}}
\hline
\textbf{Choices of $\kappa$ and $\delta$} & \textbf{Final Output} \\
\hline
$\kappa=1$ and $\delta=0$ & The standard CBO gets trapped at a local minimizer (see Figure~\ref{fig:Rastrigin}-(Left)).\\
$\kappa=1$ and $\delta=5$ & The CBO with non-degenerate diffusion successfully locates the global minimizer. \\
$\kappa=0.99$ and $\delta=0$ & The rescaled CBO converges to the global minimizer.\\
 $\kappa=0.01$ and $\delta=5$ & The rescaled CBO converges to the global minimizer (see Figure~\ref{fig:Rastrigin}-(Right)).\\
\hline
\end{tabular}
\caption{We apply the CBO dynamics~\eqref{CBOkappa particle} to the Rastrigin function $R(x)$ with different choices of $\kappa$ and $\delta$.}\label{table}
\end{table}

The advantage of the rescaled CBO \eqref{CBO kappa} manifests already in several contexts: 
\begin{itemize}
    \item In \cite{herty2025multiscale}, it ensures the validity of an \textit{averaging principle} which reduces a singularly perturbed (or, multi--scaled) dynamics of the form \eqref{CBO kappa}, into an averaged one. This can be seen as a model reduction technique for solving multi--level optimization problems using multiple populations of interacting particles. 
    \item In \cite{huang2024self}, it allows to prove existence and uniqueness of an invariant probability measure $\rho_*^\alpha$, 
    and to establish the long--time behavior of its law. Moreover, its approximation with empirical measures has been proven with an explicit rate of convergence, leading to the so-called \textit{Self--interacting CBO}. 
    \item In \cite{huang2025uniform}, it helped to guarantee a uniform--in--time propagation of chaos for the population of particles evolving along the dynamics \eqref{CBOkappa particle}. This property is crucial for ensuring a stable and reliable long--term convergence, which secure the practical effectiveness of CBO methods. 
\end{itemize}

The present manuscript is in the continuation of the former work \cite{huang2024self,huang2025uniform}, where the proof of global convergence is absent. Our main contributions are discussed next.
\vspace*{-0.6em}

\subsection{Highlights} \label{Highlights} 

The three main features of the present manuscript are as follows. 
\begin{itemize}
\item \textbf{Asymptotic behavior when $t\to\infty$ and $\alpha \to\infty$:} We obtain the first \textit{true} convergence result in the literature in the sense that both limits, $t \to \infty$ and $\alpha \to \infty$, can be taken. More specifically, we establish the asymptotic behavior of the unique invariant ($t \to \infty$) probability measure $\rho_*^\alpha$ of \eqref{CBO kappa} as $\alpha \to \infty$. 
More precisely, when the function $f(\cdot)$ to be minimized has a unique global minimizer $x_*$, we prove that  $\mathbb{W}_{2}(\rho^{\alpha}_{t}, \rho^{\alpha}_{*})^{2} \longrightarrow 0$ when $t\to \infty$, and
\begin{equation*}
    \lim_{\alpha\to \infty}\lim_{t\to \infty}\; \mathfrak{m}_{\alpha}(\rho^{\alpha}_{t})
    =
    \lim_{\alpha\to \infty}\lim_{t\to \infty} \; \kappa^{-1}\,\EE\left[\,\OX_t\,\right]=  x_*\,,
\end{equation*}
where $\OX_{\cdot}$ satisfies \eqref{CBO kappa}.
    \item \textbf{Unconditional global convergence:} We are able to bypass classical assumptions in the literature. 
    \begin{itemize}
        \item 
        Our results hold without any prior assumptions on the initial distribution $\rho_0$ of \eqref{CBO kappa} or its corresponding interacting particles system \eqref{CBOkappa particle}, provided it has sufficiently many finite moments. In particular, we do not assume that a global minimizer of $f(\cdot)$ is contained within the support of $\rho_0$, a crucial assumption in both \cite{carrillo2018analytical} and \cite{fornasier2024consensus}. Very recently, \cite{fornasier2025regularity} demonstrated that this assumption can actually be removed by instead requiring regular initial data $\rho_0 \in H^m(\mathbb{R}^d)$. 
        Or in \cite{huang2024fast}, a feedback control is used to drive the particles nearby the minimizer.
        In contrast, our approach does not use control and does not impose any regularity conditions on $\rho_0$, thereby allowing cases where $\rho_0$ is a Dirac measure, as has been considered in the self--interacting CBO model \cite{huang2024self}.
        \item No regularity assumptions on the function $f(\cdot)$ are needed, besides being positive, locally Lipschitz continuous, and with a quadratic growth (see Assumption \ref{assum1}). These assumptions were used in \cite{carrillo2018analytical} to guarantee the well-posedness of the CBO dynamics \eqref{CBO}. To establish global convergence, \cite[Assumption 4.1]{carrillo2018analytical} imposes additional conditions, such as $f \in \mathcal{C}^2(\mathbb{R}^d)$ and the boundedness of $\|\nabla^2 f\|_\infty$. In the framework of \cite[Definition 8]{fornasier2024consensus}, some additional local growth conditions on $f$ around the global minimizer are required. For our rescaled CBO we do not require these additional assumptions for the convergence proof.
    \end{itemize}
    \item \textbf{Multiple--minimizer case:} Although our rescaled CBO is not specifically designed to address the multiple minimizer problem, we can still establish its asymptotic properties in the case where the objective function $f(\cdot)$ has multiple non-trivial minimizers, whether they are discrete points or compact subsets (see Theorem \ref{thm:conver3}). 
    More precisely, denoting $\mathcal{M}$ the set of minimizers, we prove
    \begin{equation*}
    \begin{aligned}
        & \lim\limits_{\alpha\to \infty}\,\lim\limits_{t\to \infty}  \operatorname{dist}(\mathfrak{m}_{\alpha}(\rho^{\alpha}_{t}), \,\mathcal{M}) 
        =\lim\limits_{\alpha\to \infty} \, \lim\limits_{t\to\infty} \operatorname{dist}\!\left( \kappa^{-1}\,\EE\left[\,\OX_t\,\right],\,\mathcal{M}\right)
        = 0.
    \end{aligned}
    \end{equation*}
    Let us mention here that specially designed CBO models aimed at addressing multiple-minimizer problems have been introduced in \cite{fornasier2025pde, bungert2022polarized}.
\end{itemize}
\vspace*{-1em}

\subsection{Strategy and organization}\label{strategy}

The proof of our main global convergence result can be divided into three steps as described in the following. 

\underline{\textit{Step 1.}}
\begin{equation*}
\!\!\!\!\!\!\!\left.
\begin{aligned}
    & \text{\textbullet \; Uniform in $(t,\alpha)$ moment bound (Theorem \ref{thm: finite moments})}\\
    & \text{\textbullet \; Existence of invariant measure (Theorem \ref{thm: existence})} \\
    & \text{\textbullet \; Long-time limit of law + uniqueness (Theorem \ref{thm:long time})}
\end{aligned}
\; \right\}
\; \Rightarrow \;
\substack{\bm{(*)}\;                \text{\normalsize Invariant measure has uniform}\\ 
    \text{\normalsize in $\alpha$ second moment bound}\\
    \text{\normalsize (Theorem \ref{thm: inv unif bound})} }
\end{equation*}
\vspace*{-1em}

\underline{\textit{Step 2.}}
\begin{equation*}
\left.
\begin{aligned}
    & \substack{\bm{(*)}\; \text{\normalsize  Invariant measure has uniform}\\ 
    \text{\normalsize in $\alpha$ second moment bound}\\
    \text{\normalsize (Theorem \ref{thm: inv unif bound})} }\\
    & \\
    & \substack{\text{\normalsize \textbullet \; Invariant measure satisfies}\\ 
    \text{\normalsize Harnack inequality}\\
    \text{\normalsize (Corollary \ref{cor: Harnack})}} 
\end{aligned}
\; \right\}
\; \Rightarrow \;
\substack{\bm{(**)}\; \text{\normalsize Invariant measure has uniform}\\ 
    \text{\normalsize in $\alpha$ positive lower bound}\\
    \text{\normalsize (Lemma \ref{lem:mass})}}
\end{equation*}
\vspace*{-1em}

\underline{\textit{Step 3.}}
\begin{equation*}
\!\!\!\!\!\!\!\!\left.
\begin{aligned}
\substack{\bm{(**)}\; \text{\normalsize Invariant measure has uniform}\\ 
    \text{\normalsize in $\alpha$ positive lower bound}\\
    \text{\normalsize (Lemma \ref{lem:mass})}}
    \Rightarrow \;
&\substack{\text{\normalsize\; Laplace principle/concentration}\\
\text{\normalsize (Theorem \ref{thm:conver} \& Theorem \ref{thm:conver3}) }}\\
&\substack{\text{\normalsize \textbullet \; Existence of invariant measure}\\
\text{\normalsize (Theorem \ref{thm: existence})}}
\end{aligned}
\right\}
\Rightarrow \substack{\text{\normalsize \textbf{Convergence}}\\
\text{\normalsize (Corollary \ref{cor: conv to min}}\\ \text{\normalsize \& Corollary \ref{cor: conv multi})}}
\end{equation*}

Our paper is organized as follows. We start in \textbf{Section \ref{sec:useful}} by setting our main assumptions and providing some estimates, in particular on the moment bound. Then \textbf{Section \ref{sec: invariant}} shall be devoted to the invariant measure. Therein, we shall prove its existence, the long-time limit ($t\to \infty$), and other of its important properties: uniform-in-$\alpha$ second moment and Harnack's inequality. Subsequently, the limit $\alpha\to \infty$ will be tackled in \textbf{Section \ref{sec: convergence}}. In particular, we shall prove a Laplace's principle, then the global convergence towards the minimizer. An analogous convergence result will also be provided in the case of multiple non-trivial minimizers. \textbf{Section \ref{sec: num}} contains some numerical experiments and discussion regarding our model and results. 
Finally, a summary of our findings is in \textbf{Section \ref{sec: conclusion}} at the end of the manuscript, together with future research directions. Our paper concludes with \textbf{Appendix \ref{app: existence}} which contains detailed computations for verifying Theorem \ref{thm: existence}.

\section{Useful estimates}\label{sec:useful}

We start by recalling some notation. 
When a measure $\mu$ has a density $\varrho$ with respect to Lebesgue measure that we denote by $\dd x$, then $\mu$ is absolutely continuous with respect to $\text{d}x$, we write $\mu \ll \dd x$ and $\varrho=\frac{\text{d}\mu}{\text{d}x}$  is Radon--Nikodym derivative of $\mu$ with respect to $\dd x$. Unless a confusion arises, we shall use the same notation for a measure and for its density (when it exists). 
We shall also denote by $W^{k,p}(\Omega)$, $p\geq 1, k\geq 0$ the standard Sobolev space of functions whose generalized derivatives up to order $k$ are in $L^{p}(\Omega)$, and we denote by $W^{k,p}_{loc}(\Omega)$ the class of functions $f$ such that $\chi f\in W^{k,p}(\Omega)$ for each $\chi \in \mc{C}_c^\infty (\Omega)$ the class of infinitely differentiable functions with compact support in $\Omega$.

In what follows, $\|\cdot\|$ denotes the Frobenius norm of a matrix and $|\cdot|$ is the standard Euclidean norm in $\RR^d$; $\mathscr{P}(\RR^d)$ denotes the space of probability measures on $\RR^d$, and $\mathscr{P}_p(\RR^d)$ with $p\geq 1$ contains all $\mu\in \mathscr{P}(\RR^d)$ such that $\mu(|\cdot|^p):=\int_{\RR^d}|x|^p\mu(\dx)<\infty$; it is equipped with $p$--Wasserstein distance $\mathbb{W}_p(\cdot,\cdot)$. 

Our main assumption concerns the function $f(\cdot)$ to be minimized, and can be stated as follows.  

\begin{assum}\label{assum1}
	We assume the following properties for the objective function.
	\begin{enumerate}
		\item $f: \RR^d\to \RR$ is bounded from below by $\underline f := \min f$, and there exists some constant $L_{f}>0$ such that
		\begin{equation*}
			|f(x)-f(y)|\leq L_f(1+|x|+|y|)|x-y|, \quad \forall x,y\in \RR^d\,.
		\end{equation*}
		\item There exist constants $c_\ell,c_u>0$  and $M>0$ such that
		\begin{equation*}
		 f(x)-\underline f\leq c_u(|x|^2+1),\quad \forall x\in \RR^d \; \mbox{ and } \quad c_\ell|x|^2\leq f(x)-\underline f \;\mbox{ for }|x|\geq M\,.
		\end{equation*}
	\end{enumerate}
\end{assum}

We note that the bound $f(x)-\underline f\leq c_u(|x|^2+1)$ is an immediate consequence of the Lipschitz growth condition in assumption (1). We retain it here to make the dependence on the constant $c_u$ explicit in the subsequent results. 

\begin{remark}
The conditions described in Assumption \ref{assum1} are satisfied in practice by a wide range of functions. 
\begin{itemize}
    \item Local Lipschitz growth:   
        This condition is quite standard and guarantees, in particular, the well-posedness of the dynamics (the associated SDE).
    
        \item Quadratic growth outside a ball:   
        This assumption is satisfied by essentially any objective function whose minimizer (or minimizers) lie within a known bounded region. In practice, optimization procedures often restrict the search domain to a closed ball. One may therefore truncate the objective function on this ball and extend it with a quadratic term outside, e.g. $f(x) \leftarrow f(x)\mathds{1}_{B}(x) + \bigl(|x|^{2} + C\bigr)\mathds{1}_{B^{c}}(x)$, for some appropriate constant $C$. With this modification, all polynomial objective functions, among many others, fall within our framework.
\end{itemize}
\end{remark}

The mean--field term $\mathfrak{m}_{\alpha}(\mu)$ plays a crucial role in the behavior of the dynamics. Its properties are essential for our analysis to be carried out. Thus we shall recall an estimate on $\mathfrak{m}_{\alpha}(\mu)$ that we borrow from \cite[Lemma 3.3]{carrillo2018analytical}.

\begin{lemma}\label{lem: useful estimates}
Let $f(\cdot)$ satisfy Assumption \ref{assum1}, then it holds that
\begin{equation}\label{lemeq2}
    |\mathfrak{m}_{\alpha}(\nu)|^2\leq \frac{\int_{\RR^d}|x|^2e^{-\alpha f(x)}\nu(\dd x)}{\int_{\RR^d}e^{-\alpha f(x)}\nu(\dd x)}  \leq b_1+b_2\int_{\RR^d}|x|^2 \dd\nu \quad \forall \nu\in \mathscr{P}_2(\mathbb{R}^d)\,,
\end{equation}
where 
\begin{equation}\label{constant b1 b2}
    b_1=M^2+b_2 \quad \text{ and } \quad b_2 = 2\,\frac{c_u}{c_\ell}\left(1+\frac{1}{\alpha \, M^2 \, c_\ell}\right).
\end{equation}
\end{lemma}

\begin{remark}
    In fact, as it has been already observed in \cite[Remark.1]{carrillo2018analytical}, the constant $b_2$ can be chosen to be independent of $\alpha$ for $\alpha\geq 1$.
\end{remark}

\begin{lemma}\label{thmexist}  
	Let $f(\cdot)$ satisfy Assumption \ref{assum1} and $\rho_0\in\mathscr{P}_{p}(\RR^d)$ for any $p\geq 2$, then the dynamics \eqref{CBO kappa} has a unique strong solution and satisfies
	\begin{equation}\label{mbound}
       \sup_{t\in[0,T]} \,\EE\left[\,|\OX_t|^{p}\,\right]<\infty \quad \quad \forall\, 0<T < \infty.
	\end{equation}
\end{lemma} 

\begin{proof}
The proof of well--posedness for the particle system \eqref{CBOkappa particle} follows the same reasoning as in \cite[Theorem 2.1]{carrillo2018analytical} or \cite[Theorem 2.2]{gerber2023mean}, relying on the fact that the coefficients in \eqref{CBOkappa particle} are locally Lipschitz. 
For the mean-field dynamics \eqref{CBO kappa}, well-posedness can be established using the Leray-Schauder fixed-point theorem, as demonstrated in \cite[Theorem 3.1]{carrillo2018analytical} or \cite[Theorem 2.2]{gerber2023mean}. 
Finally, the moment estimates in \eqref{mbound} are derived from \cite[Lemma 3.4]{carrillo2018analytical} or \cite[Lemma 3.5]{gerber2023mean} by using the Burkholder--Davis--Gundy inequality.
\end{proof}

We can now prove the following uniform--in--time second moment bound for $\OX_t$ which will play a key role in the sequel.

\begin{theorem}\label{thm: finite moments}
Let $f(\cdot)$ satisfy assumption \ref{assum1} and  $\rho_0\in\mathscr{P}_{4}(\RR^d)$. Consider   $(\OX_t)_{t\geq 0}$  with $\rho_0$-distributed initial data satisfying \eqref{CBO kappa}. Then  for sufficiently small $\kappa$ there exists some constant $C_2$ depending only $\lambda,\sigma,d,\delta,\kappa $ and $f$  such that it holds 
\begin{equation}\label{eq:uni}
    \sup_{t\geq 0}\;\EE\left[\,|\OX_t|^2\,\right]\leq \mathbb{E}\left[\, |\OX_{0}|^{2}\, \right] + C_2\,.
\end{equation}
\end{theorem}

Before we proceed with the proof, let us recall It\^{o}'s formula of the sake of self-containedness: denote the dynamics of the process $\overline{X}$ by the shorthand notation
\begin{equation*}
    \dd \overline{X}_{t} = \bm{\mu}_{t} \, \dd t + \bm{\varsigma}_{t} \,\dd B_t
\end{equation*}
It\^{o}'s formula states that for any function $h\in C^{2}(\mathbb{R}^{d})$ one has
\begin{equation}\label{ito}
\begin{aligned}
    \dd h(\overline{X}_t) 
    & = \left\{\,\nabla h(\overline{X}_t)\cdot\bm{\mu}_{t} + \frac{1}{2}\text{trace}(\bm{\varsigma}_{t}\cdot D^{2}h(\overline{X}_t)\,\bm{\varsigma}_{t}) \,\right\} \dd t \\
    & \qquad  + \nabla h(\overline{X}_t)\cdot\bm{\varsigma}_{t}\,\dd B_t.
\end{aligned}
\end{equation}

\begin{proof}
Let $h(x):=\frac{1}{2}|x|^{2}$, for $x\in \mathbb{R}^{d}$, and let $\overline{X}$ be the process defined in \eqref{CBO kappa}. Then applying It\^{o}'s formula \eqref{ito} leads to
\begin{align*}
    \rd h(\OX_{t}) & = \bigg\{ -\lambda \nabla h(\OX_{t})\cdot \big(\OX_{t} - \kappa \, \mathfrak{m}_{\alpha}(\rho_{t})\big) \\
    & \quad \quad \quad + \frac{\sigma^{2}}{2} \text{trace}\left[D^{2}h(\OX_{t}) \bigg( \delta\,\mathds{I}_{d} + D\big(\OX_{t} - \kappa\,\mathfrak{m}_{\alpha}(\rho_{t})\big) \bigg)^{2} \right] \bigg\}\,\rd t \\
    & \quad \quad \quad \quad \quad \quad + \sigma\,\nabla h(\OX_{t})^{\top}\bigg( \delta\,\mathds{I}_{d} + D\big(\OX_{t} - \kappa\,\mathfrak{m}_{\alpha}(\rho_{t})\big) \bigg)\, \rd B_{t}\\
    & =: F(t,\OX_{t})\,\rd t + G(t,\OX_{t})\,\rd B_{t}.
\end{align*}

One could observe that, thanks to Lemma \ref{thmexist} we have
\begin{equation*}
    \mathbb{E}\left[\int_{0}^{t} G(s,\OX_{s})\,\rd B_{s}\right] = 0,\quad \forall\, t>0.
\end{equation*}
This is true in particular because of the well-posedness of the process $t \mapsto \OX_{t}$ according to Lemma \ref{thmexist} with $X_0^i\in \mathscr{P}_{4}(\RR^d)$, and its finite moments which guarantee that $\mathbb{E}\left[\int_{0}^{t} |G(s,\OX_{s})|^{2}\rd s\right]\leq C\,\EE\left[\int_0^t|\OX_s|^{4}\rd s\right] <+\infty$. 
Therefore we have 
\begin{equation}\label{differentiation}
    \frac{\rd}{\rd t}\mathbb{E}\left[\frac{1}{2}|\OX_{t}|^{2}\right] = \mathbb{E}[F(t, \OX_{t})], \quad \forall \, t>0.
\end{equation} 

We can now further estimate the right-hand side as follows. Recall
\begin{align*}
    F(t,\OX_{t}) 
     &= -\lambda \nabla h(\OX_{t})\cdot \big(\OX_{t} - \kappa \, \mathfrak{m}_{\alpha}(\rho^{N}_{t})\big) \\
     &\quad+ \frac{\sigma^{2}}{2} \text{trace}\left[D^{2}h(\OX_{t}) \bigg( \delta\,\mathds{I}_{d} + D\big(\OX_{t} - \kappa\,\mathfrak{m}_{\alpha}(\rho^{N}_{t})\big) \bigg)^{2} \right]
\end{align*}
and with $h(x)=\frac{1}{2}|x|^{2}$ we have $\nabla h(x) = x$,  and $D^{2}h(x) = \mathds{I}_{d}$. 
We deal first with the trace term in $F(t,\OX_{t})$. We have
\begin{align*}
    & \text{trace}\left[D^{2}h(\OX_{t}) \bigg( \delta\,\mathds{I}_{d} + D\big(\OX_{t} - \kappa\,\mathfrak{m}_{\alpha}(\rho_{t})\big) \bigg)^{2} \right]\\
    &\quad\quad=\; \text{trace}\left[D^{2}h(\OX_{t}) \bigg( \delta^{2}\,\mathds{I}_{d} + D\big(\OX_{t} - \kappa\,\mathfrak{m}_{\alpha}(\rho_{t})\big)^{2} + 2\delta\, D\big(\OX_{t} - \kappa\,\mathfrak{m}_{\alpha}(\rho_{t})\big)\bigg) \right]\\
    &\quad\quad=\; \delta^{2}\,\text{trace}(D^{2}h(\OX_{t})) + \text{trace}\big[ D^{2}h(\OX_{t})D(\OX_{t} - \kappa\,\mathfrak{m}_{\alpha}(\rho_{t}))^{2}\big] \\
    & \quad\quad\quad\quad + 2\delta\,\text{trace}\big[D^{2}h(\OX_{t})D(\OX_{t} - \kappa\,\mathfrak{m}_{\alpha}(\rho_{t}))\big]\;\\
    &\quad\quad\leq \; \delta^{2}d +  \sum\limits_{k=1}^{d} |(\OX_{t} - \kappa\,\mathfrak{m}_{\alpha}(\rho_{t}))_{k}|^{2}+ 2\delta \sum\limits_{k=1}^{d}|(\OX_{t} - \kappa\, \mathfrak{m}_{\alpha}(\rho_{t}))_{k}|\\
    &\quad\quad\leq \;  \delta^{2}d +\delta d+(\delta+1)|\OX_{t} - \kappa\, \mathfrak{m}_{\alpha}(\rho_{t})|^{2}.
\end{align*}
In the last line, we used $|(\OX_{t} - \kappa\, \mathfrak{m}_{\alpha}(\rho_{t}))_{k}| \leq \frac{1}{2}(|(\OX_{t} - \kappa\, \mathfrak{m}_{\alpha}(\rho_{t}))_{k}|^{2} + 1)$, then the sum yields $\sum\limits_{k=1}^{d}|(\OX_{t} - \kappa\, \mathfrak{m}_{\alpha}(\rho_{t}))_{k}| \leq \frac{1}{2}(|\OX_{t} - \kappa\, \mathfrak{m}_{\alpha}(\rho_{t})|^{2} + d)$. 
Therefore we have
\begin{align*}
    F(t,\OX_{t}) 
    & \leq -\lambda\left(1 -\frac{\kappa}{2}\right) |\OX_{t}|^{2} + \lambda\,\frac{\kappa}{2}|\mathfrak{m}_{\alpha}(\rho_{t})|^{2} \\
    & \quad \quad \frac{\sigma^2}{2}\delta(\delta+1)d+\sigma^2(\delta+1)|\OX_{t}|^2+\sigma^2(\delta+1) \kappa^2 |\mathfrak{m}_{\alpha}(\rho_{t})|^{2}
     \\
    & \leq \left\{ -\lambda\left(1 -\frac{\kappa}{2}\right) + \sigma^{2}\left(1+\delta\right) \right\}\,|\OX_{t}|^{2}  \\
    & \quad \quad + \left\{ \lambda\,\frac{\kappa}{2} + \sigma^{2}\left(1+\delta\right)\kappa^{2} \right\}\,|\mathfrak{m}_{\alpha}(\rho_{t})|^{2} + \frac{\sigma^2}{2}\delta(\delta+1)d.
\end{align*}
Recalling \eqref{lemeq2}, it holds that
\begin{equation*}
|\mathfrak{m}_{\alpha}(\rho_{t})|^{2}\leq b_1+b_2\,\EE[|\OX_t|^2]\,,
\end{equation*}
which implies that
\begin{align*}
   \EE[F(t,\OX_{t})]  & \leq \left\{ -\lambda\left(1 -\frac{\kappa}{2}\right) + \sigma^{2}\left(1+\delta\right) +b_2\left( \lambda\,\frac{\kappa}{2} + \sigma^{2}\left(1+\delta\right)\kappa^{2} \right)\right\}\,\EE[|\OX_{t}|^{2}] \\
    & \quad \quad + \left( \lambda\,\frac{\kappa}{2} + \sigma^{2}\left(1+\delta\right)\kappa^{2} \right)b_1 + \frac{\sigma^2}{2}\delta(\delta+1)d.  
\end{align*}

The latter, together with \eqref{differentiation} finally yields
\begin{align*}
    \frac{\rd}{\rd t}\mathbb{E}\left[\frac{1}{2}|\OX_{t}|^{2}\right] & = \mathbb{E}\left[F(t,\OX_{t})\right]
     \leq -\gamma\,\mathbb{E}\left[|\OX_{t}|^{2}\right] + C_2
\end{align*}
where
\begin{equation}\label{gamma C}
\begin{aligned}
    & \gamma :=  \lambda\left(1 -\frac{\kappa}{2}\right) - \sigma^{2}\left(1+\delta \right) - b_2\left( \lambda\,\frac{\kappa}{2} + \sigma^{2}\left(1+\delta\right)\kappa^{2} \right) \\
    & \text{and } \quad C_2 := \left( \lambda\,\frac{\kappa}{2} + \sigma^{2}\left(1+\delta\right)\kappa^{2} \right)b_1 + \frac{\sigma^2}{2}\delta(\delta+1)d. 
\end{aligned}
\end{equation}
It suffices to choose $\lambda,\sigma, \delta$, and $\kappa$ in order to guarantee $\gamma>0$,
which together with Gr\"{o}nwall's inequality yields
\begin{equation}\label{eq:uni proof}
    \mathbb{E}\left[ |\OX_{t}|^{2} \right] \leq \mathbb{E}\left[ |\OX_{0}|^{2} \right]\, e^{-\gamma\,t } + C_2 \, \leq \mathbb{E}\left[ |\OX_{0}|^{2} \right] + C_2.
\end{equation}
\end{proof}

\begin{remark}\label{rmk: configuration}
An example of a simple configuration of the parameters where \eqref{eq:uni} is guaranteed can be obtained as follows. 
Recalling the notation in \eqref{constant b1 b2} and then in \eqref{gamma C}, we can choose for example $\delta=1$, then we need to satisfy the inequality
\begin{align*}
    & \gamma =  \lambda\left(1 -\frac{\kappa}{2}\right) - 2\,\sigma^{2} - b_2\left( \lambda\,\frac{\kappa}{2} + 2\sigma^{2}\kappa^{2} \right) \; >0\\%\frac{\sigma^{2}}{2}\left\{ (d+p-2)\delta^{2} + d(p-1)\delta\right\}\\
    \Leftrightarrow \quad & \lambda \; > \lambda \, \frac{\kappa}{2} (1+b_2) + 2\, \sigma^{2}(1+ b_2 \kappa^{2} ) \\
    \Leftarrow \quad & \lambda \; > \lambda \, \kappa \,(1+b_2) + 4\, \sigma^{2}(1+ b_2 \kappa^{2} )\\
    \Leftrightarrow \quad & \lambda \; (1 - \, \kappa \,(1+b_2)) > 4\, \sigma^{2}(1+ b_2 \kappa^{2} )
\end{align*}
Choosing $\kappa < \frac{1}{2}(1+ b_2)^{-1}$ yields in particular $(1 - \, \kappa \,(1+b_2)) > \frac{1}{2}$ and  $b_2 \,\kappa^{2} < 1$. Therefore, we have 
\begin{align*}
    \lambda \; (1 - \, \kappa \,(1+b_2)) > \frac{\lambda}{2} 
    \quad  \text{ and } \quad  8\, \sigma^{2} > 4\, \sigma^{2}(1+ b_2 \kappa^{2} ).
\end{align*}
Conclusion: to guarantee $\gamma>0$, an example of sufficient conditions is when
\begin{align*}
    \delta=1, \quad \kappa < \frac{1}{2(1+b_2)}, \quad  \text{ and } \quad \lambda > 16\, \sigma^{2}.
\end{align*}
\end{remark}

Lastly, we recall a stability estimate on $\mathfrak{m}_{\alpha}(\mu)$ from \cite[Corollary 3.3, Proposition A.3]{gerber2023mean}. It will be used in the proof of the subsequent results. 

\begin{lemma}\label{lem: stab}
Suppose $f(\cdot)$ satisfies Assumption \ref{assum1}. Then for all $R>0$,
there exists some constant $L_{\mathfrak{m}}>0$ depending on $R$ such that
\begin{equation}\label{ineq stab}
    |\mathfrak{m}_{\alpha}(\mu)-\mathfrak{m}_{\alpha}(\nu)|\leq L_{\mathfrak{m}} \mathbb{W}_2(\mu,\nu)\quad \forall (\mu,\nu)\in \mathscr{P}_{2,R}(\RR^d)\times \mathscr{P}_{2}(\RR^d)\,.
\end{equation}
\end{lemma}

\section{On the invariant measure}\label{sec: invariant}

This section is devoted to the limit $t\to \infty$. We will first prove existence of an invariant measure for \eqref{CBO kappa} in \S\ref{subsec: exist}. Then we will we will establish in \S\ref{subsec: limit t} a $\mathbb{W}_{2}$--exponential contraction for the convergence of the law of \eqref{CBO kappa} towards the invariant measure. This shall be next used in \S \ref{subsec: unif} where we will provide a uniform-in-$\alpha$ second moment of the invariant measure (that will be later needed in Section \ref{sec: convergence}). Finally, we will show that the invariant measure verifies a Harnack type inequality which will lead to a positive and uniform-in-$\alpha$ lower bound on its support: a key result for Laplace's principle in Section \ref{sec: convergence}.

\subsection{Existence}\label{subsec: exist}

The following result ensures existence of an invariant measure for the process \eqref{CBO kappa}.

\begin{theorem}\label{thm: existence}
    Suppose $f(\cdot)$ satisfies Assumption \ref{assum1}. For sufficiently small $\kappa \in (0,1)$, and for a large $\lambda>0$, there exists an invariant measure $\rho^{\alpha}_{*}\in \mathscr{P}_{2,R}(\RR^d)$, for some sufficiently large $R$ that depends on $\kappa, \delta, \lambda, \sigma, b_1, b_2$ and the initial distribution. 
\end{theorem}

\begin{proof}
The proof relies on \cite[Theorem 2.2]{zhang2023existence}. See Appendix \ref{app: existence}.
\end{proof}

\begin{remark}\label{rmk: configuation 2}
As can be seen in Appendix \ref{app: existence}, the precise condition on the parameters is explicitly given by \eqref{equiv C1 C3} that is 
\begin{equation*}
    2\lambda - 4  > \lambda\,\kappa\,(1+b_2) + 4\,\kappa^{2} b_2
\end{equation*}
where we recall $b_2$ is defined in \eqref{constant b1 b2}. 
A simple example of sufficient conditions is when
\begin{equation*}
    \kappa < \frac{1}{2(1+b_2)} \qquad \text{ and } \qquad \lambda >4.
\end{equation*}
\end{remark}

\subsection{Long-time limit}\label{subsec: limit t}

The next result provides a $\mathbb{W}_{2}$--exponential contraction for the evolution of the distribution of $\overline{X}$ as defined in \eqref{CBO kappa} and cements its long-time limit.

\begin{theorem}\label{thm:long time}
Let us consider \eqref{CBO kappa} whose distribution is $\rho^{\alpha}_t := \operatorname{Law}(\overline{X}_t)$ for all $t\geq0$. Suppose $f(\cdot)$ satisfies Assumption \ref{assum1}. Then for sufficiently small $\kappa \in (0,1)$ and $\sigma>0$, there exists $\Theta>0$ such that
\begin{equation}\label{contraction}
    \mathbb{W}_{2}(\rho^{\alpha}_{t}, \rho^{\alpha}_{*})^{2} \leq e^{-\Theta \, t}\,\mathbb{W}_{2}(\rho^{\alpha}_{0}, \rho^{\alpha}_{*})^{2} \quad \forall\,t\geq 0
\end{equation}
and $\Theta$ depends on $\lambda, \sigma, \kappa$ and $L_{\mathfrak{m}}$. In particular, the invariant measure $\rho^{\alpha}_{*}$ of \eqref{CBO kappa} is unique.
\end{theorem}

We refer to Remark \ref{rmk: configuration 3} below for a precise definition of the smallness of the parameters. 

\begin{remark}
The standard CBO \eqref{CBO} (i.e. when $\kappa =1$ and $\delta =0$ in \eqref{CBO kappa}) admits infinitely many invariant measures: see \cite[Remark 1.1]{huang2025uniform}.
\end{remark}

An immediate consequence of Theorem \ref{thm:long time} and Lemma \ref{lem: stab} is the following corollary. 

\begin{corollary}\label{cor: long time of m}
In the situation of Theorem \ref{thm:long time}, it holds
\begin{equation*}
    |\mathfrak{m}_{\alpha}(\rho^{\alpha}_{t})-\mathfrak{m}_{\alpha}(\rho^{\alpha}_{*})| \leq L_{\mathfrak{m}} e^{-\frac{\Theta}{2} \, t}\,\mathbb{W}_{2}(\rho^{\alpha}_{0}, \rho^{\alpha}_{*}) \quad \forall\,t\geq 0.
\end{equation*}
In particular, one has $\lim\limits_{t\to \infty} \mathfrak{m}_{\alpha}(\rho^{\alpha}_{t}) = \mathfrak{m}_{\alpha}(\rho^{\alpha}_{*})$.
\end{corollary}

\begin{proof}[Proof of Theorem \ref{thm:long time}] 
Let us consider two processes governed by the same dynamics, but starting from different initially distributed positions. The first one is \eqref{CBO kappa} which we recall is
\begin{equation}\label{CBO kappa proof}
\begin{aligned}
    \rd \OX_t = & -\lambda\left(\OX_t - \kappa\,\mathfrak{m}_{\alpha}(\rho^{\alpha}_t)\right)\,\dt \,+ \,\sigma\left(\delta\,\mathds{I}_{d} + D\left(\OX_t - \kappa\, \mathfrak{m}_{\alpha}(\rho^{\alpha}_t)\right) \right)\,\rd B_t\\
    & \text{with } \; \rho^{\alpha}_t \, := \, \text{Law}(\OX_t) \quad \text{ and } \quad  \OX_0 \sim \rho^{\alpha}_0.
\end{aligned}
\end{equation}
We also have, thanks to Theorem \ref{thm: finite moments}, $\operatorname{Law}(X_t) = \rho^{\alpha}_{t}\in \mathscr{P}_{2}(\mathbb{R}^{d})$ for all $t\geq 0$. 

The second process is the stationary counter-part of \eqref{CBO kappa proof}, whose initial condition has the invariant measure as distribution, and that is
\begin{equation}\label{CBO kappa stat}
\begin{aligned}
    \rd \overline{Z}_t = & -\lambda\left(\overline{Z}_t - \kappa\,\mathfrak{m}_{\alpha}(\rho^{\alpha}_*)\right)\,\dt \,+ \,\sigma\left(\delta\,\mathds{I}_{d} + D\left(\overline{Z}_t - \kappa\, \mathfrak{m}_{\alpha}(\rho^{\alpha}_*)\right) \right)\,\rd B_t\\
    & \text{with } \; \rho^{\alpha}_*\, := \, \text{Law}(\overline{Z}_t) \quad \forall\, t\geq 0.
\end{aligned}
\end{equation}
Indeed, by definition of $\rho^{\alpha}_{*}$ being an invariant measure for \eqref{CBO kappa proof}, the law of $\overline{Z}_{t}$ remains equal to $\rho^{\alpha}_{*}$ for all $t\geq 0$, hence the process is stationary. 

We have 
\begin{align*}
    \rd (\overline{X}_{t} - \overline{Z}_{t}) & = - \lambda \bigg((\overline{X}_{t} - \overline{Z}_{t}) - \kappa \big(\mathfrak{m}_{\alpha}(\rho^{\alpha}_{t}) - \mathfrak{m}_{\alpha}(\rho^{\alpha}_{*})\big)\bigg)\,\rd t\\
    & \quad \quad \quad + \sigma\bigg( D(\overline{X}_{t} - \kappa\, \mathfrak{m}_{\alpha}(\rho^{\alpha}_{t})) - D(\overline{Z}_{t} - \kappa\,\mathfrak{m}_{\alpha}(\rho^{\alpha}_{*})) \bigg)\, \rd B_{t}.
\end{align*}
As in the proof of Theorem \ref{thm: finite moments}, we use the function $h:\mathbb{R}^{d}\to \mathbb{R}$, such that $h(z) = \frac{1}{2}|z|^{2}$, and we apply It\^o's formula \eqref{ito}  which  leads to 
\begin{align*}
    \rd h(\overline{X}_{t} - \overline{Z}_{t}) = 
    & \bigg\{ -\lambda (\overline{X}_{t} - \overline{Z}_{t})\cdot \big[(\overline{X}_{t} - \overline{Z}_{t}) - \kappa(\mathfrak{m}_{\alpha}(\rho^{\alpha}_{t}) - \mathfrak{m}_{\alpha}(\rho^{\alpha}_{*}))\big]\\
    & \quad + \frac{\sigma^{2}}{2}\,\text{trace}\bigg[\bigg( D(\overline{X}_{t} - \kappa\, \mathfrak{m}_{\alpha}(\rho^{\alpha}_{t})) - D(\overline{Z}_{t} - \kappa\,\mathfrak{m}_{\alpha}(\rho^{\alpha}_{*})) \bigg)^{2}\bigg]\bigg\}\,\rd t\\
    & \quad \quad + \sigma (\overline{X}_{t} - \overline{Z}_{t})\cdot \bigg[D(\overline{X}_{t} - \kappa\, \mathfrak{m}_{\alpha}(\rho^{\alpha}_{t})) - D(\overline{Z}_{t} - \kappa\,\mathfrak{m}_{\alpha}(\rho^{\alpha}_{*}))\bigg]\,\rd B_{t}\\
    & =: \mathbf{F}_{t}\,\rd t + \mathbf{G}_{t}\, \rd B_{t}.
\end{align*}
The same arguments as in the beginning of the proof of \cite[Theorem 2.4]{huang2025uniform} (see in particular the proof of \cite[equation (2.6)]{huang2025uniform}), allow us to write
\begin{align*}
    & \frac{\rd}{\rd t}\EE\left[|\overline{X}_{t} - \overline{Z}_{t}|^{2}\right]  = 2 \, \EE\left[\mathbf{F}_{t}\right]\\
    & \quad = -2\,\lambda\, \EE\left[|\overline{X}_{t} - \overline{Z}_{t}|^{2}\right] + 2\,\lambda\,\kappa\, \EE\left[(\overline{X}_{t} - \overline{Z}_{t})\cdot \big(\mathfrak{m}_{\alpha}(\rho^{\alpha}_{t}) - \mathfrak{m}_{\alpha}(\rho^{\alpha}_{*})\big)\right]\\
    & \quad \quad \quad \quad + \sigma^{2}\,\EE\left[\text{trace}\bigg(\big( D(\overline{X}_{t} - \kappa\, \mathfrak{m}_{\alpha}(\rho^{\alpha}_{t})) - D(\overline{Z}_{t} - \kappa\,\mathfrak{m}_{\alpha}(\rho^{\alpha}_{*})) \big)^{2}\bigg)\right]. 
\end{align*}
We can now estimate each of the terms above. We have
\begin{align*}
    & \EE\left[(\overline{X}_{t} - \overline{Z}_{t})\cdot \big(\mathfrak{m}_{\alpha}(\rho^{\alpha}_{t}) - \mathfrak{m}_{\alpha}(\rho^{\alpha}_{*})\big)\right] \leq \frac{1}{2}\EE\left[|\overline{X}_{t} - \overline{Z}_{t}|^{2}\right] + \frac{1}{2}|\mathfrak{m}_{\alpha}(\rho^{\alpha}_{t}) - \mathfrak{m}_{\alpha}(\rho^{\alpha}_{*})|^{2}
\end{align*}
and
\begingroup\allowdisplaybreaks
\begin{align*}
    & \EE\left[\text{trace}\bigg(\big( D(\overline{X}_{t} - \kappa\, \mathfrak{m}_{\alpha}(\rho^{\alpha}_{t})) - D(\overline{Z}_{t} - \kappa\,\mathfrak{m}_{\alpha}(\rho^{\alpha}_{*})) \big)^{2}\bigg)\right] \\
    & \quad =  \sum\limits_{k=1}^{d} \EE\left[
    \bigg( |\{\overline{X}_{t} - \kappa\,\mathfrak{m}_{\alpha}(\rho^{\alpha}_{t})\}_{k}| - |\{\overline{Z}_{t} - \kappa\,\mathfrak{m}_{\alpha}(\rho^{\alpha}_{*})\}_{k}| \bigg)^{2}
    \right]\\
    & \quad \leq \sum\limits_{k=1}^{d} \EE\left[
    \left|\big\{\overline{X}_{t}-\overline{Z}_{t}\big\}_{k} - \kappa\,\{\mathfrak{m}_{\alpha}(\rho^{\alpha}_{t}) - \mathfrak{m}_{\alpha}(\rho^{\alpha}_{*})\big\}_{k}\right| ^{2}
    \right]\\
    & \quad \leq 2 \sum\limits_{k=1}^{d} \EE\left[
    \left|\big\{\overline{X}_{t}-\overline{Z}_{t}\big\}_{k}\right|^{2} + \kappa^{2}\,\left|\{\mathfrak{m}_{\alpha}(\rho^{\alpha}_{t}) - \mathfrak{m}_{\alpha}(\rho^{\alpha}_{*})\big\}_{k}\right| ^{2}
    \right]\\
    & \quad \leq 2\, \EE\left[|\overline{X}_{t} - \overline{Z}_{t}|^{2}\right] + 2 \,\kappa^{2}\,|\mathfrak{m}_{\alpha}(\rho^{\alpha}_{t}) - \mathfrak{m}_{\alpha}(\rho^{\alpha}_{*})|^{2} .
\end{align*}
\endgroup
Therefore we have
\begin{equation}\label{eq: estimate final proof}
\begin{aligned}
    \frac{\rd}{\rd t}\EE\left[|\overline{X}_{t} - \overline{Z}_{t}|^{2}\right]  & \leq (-2\,\lambda + \lambda\,\kappa\, +2\,\sigma^{2})\, \EE\left[|\overline{X}_{t} - \overline{Z}_{t}|^{2}\right] \\
    &\quad \quad \quad  + (\lambda\,\kappa\, + 2\,\sigma^{2}\,\kappa^{2})\,|\mathfrak{m}_{\alpha}(\rho^{\alpha}_{t}) - \mathfrak{m}_{\alpha}(\rho^{\alpha}_{*})|^{2}.
\end{aligned}
\end{equation}

For sufficiently small values of $\kappa$ and $\sigma$, one guarantees $-2\,\lambda + \lambda\,\kappa\, +2\,\sigma^{2} < 0$. Let us define  $\theta:= 2\,\lambda - \lambda\,\kappa\, -2\,\sigma^{2}>0$, and $\gamma := \lambda\,\kappa\, + 2\,\sigma^{2}\,\kappa^{2}$. 
Therefore, Gr\"{o}nwall's inequality yields 
\begin{equation}\label{ineq: after gronwall}
    \mathbb{E}\left[\,|\overline{X}_{t} - \overline{Z}_{t}|^{2}\,\right] \leq e^{-\theta\, t}\mathbb{E}\left[\,|\overline{X}_{0} - \overline{Z}_{0}|^{2}\,\right] + \gamma \int_{0}^{t} |\mathfrak{m}_{\alpha}(\rho^{\alpha}_{s}) - \mathfrak{m}_{\alpha}(\rho^{\alpha}_{*})|^{2}\, e^{-\theta(t-s)}\,\dd s.
\end{equation}
Using $\mathbb{W}_{2}(\mu,\nu)^{2} = \inf\limits_{(Y\sim \mu,\, Y'\sim \nu)} \mathbb{E}[|Y-Y'|^{2}]$, we can take the infimum first in the left-hand side of \eqref{ineq: after gronwall}, then in the right-hand side, and obtain
\begin{equation*}
    \mathbb{W}_{2}(\rho^{\alpha}_{t}, \rho^{\alpha}_{*})^{2} \leq e^{-\theta\, t}\mathbb{W}_{2}(\rho^{\alpha}_{0}, \rho^{\alpha}_{*})^{2} + \gamma \int_{0}^{t} |\mathfrak{m}_{\alpha}(\rho^{\alpha}_{s}) - \mathfrak{m}_{\alpha}(\rho^{\alpha}_{*})|^{2}\, e^{-\theta(t-s)}\,\dd s.
\end{equation*}
We are left with the last term which we estimate as follows. Since $\rho^{*}_{\alpha}\in \mathscr{P}_{2,R}(\mathbb{R}^{d})$, we can use \eqref{ineq stab} and obtain
\begin{equation*}
    |\mathfrak{m}_{\alpha}(\rho^{\alpha}_{t})-\mathfrak{m}_{\alpha}(\rho^{\alpha}_{*})|\leq L_{\mathfrak{m}} \mathbb{W}_2(\rho^{\alpha}_{t},\rho^{\alpha}_{*}).
\end{equation*}
Subsequently, we have
\begin{equation*}
\begin{aligned}
    \mathbb{W}_{2}(\rho^{\alpha}_{t}, \rho^{\alpha}_{*})^{2} 
    & \leq e^{-\theta\, t}\mathbb{W}_{2}(\rho^{\alpha}_{0}, \rho^{\alpha}_{*})^{2} + \gamma \, L_{\mathfrak{m}}^{2} \,
    \int_{0}^{t} \mathbb{W}_2(\rho^{\alpha}_{s},\rho^{\alpha}_{*})^{2}\, e^{-\theta(t-s)}\,\dd s\\
    \Rightarrow\; e^{\theta\, t} \mathbb{W}_{2}(\rho^{\alpha}_{s}, \rho^{\alpha}_{*})^{2}& \leq \mathbb{W}_{2}(\rho^{\alpha}_{0}, \rho^{\alpha}_{*})^{2} + \gamma \, L_{\mathfrak{m}}^{2} \,
    \int_{0}^{t} \mathbb{W}_2(\rho^{\alpha}_{s},\rho^{\alpha}_{*})^{2}\, e^{\theta\,s}\,\dd s
\end{aligned}
\end{equation*}
Let us introduce $u(t): = e^{\theta\, t}\mathbb{W}_{2}(\rho^{\alpha}_{t}, \rho^{\alpha}_{*})^{2}$, $a := \gamma \, L_{\mathfrak{m}}^{2}$. Then we have
\begin{equation*}
    u(t) \leq u(0) + \int_{0}^{t} a \,u(s)\, \dd s
\end{equation*}
We can now apply the integral version\footnote{Suppose $u(t) \leq u_{0} + \int_{0}^{t} a(s)u(s)\,\dd s + \int_{0}^{t} b(s)\,\dd s$ with $a\geq 0$. Then denoting by $A(t) = \int_{0}^{t} a(s)\,\dd s$, the following inequality holds: $u(t) \leq u_{0} e^{A(t)} + \int_{0}^{t} b(s)\,e^{A(t) - A(s)}\,\dd s$.} of Gr\"{o}nwall's inequality and obtain 
\begin{equation*}
\begin{aligned}
    u(t) 
    & \leq e^{a\,t} u(0)\\
    \Rightarrow\; \mathbb{W}_{2}(\rho^{\alpha}_{t}, \rho^{\alpha}_{*})^{2} &\leq e^{(a-\theta) t}\,\mathbb{W}_{2}(\rho^{\alpha}_{0}, \rho^{\alpha}_{*})^{2} \quad \forall\,t\geq 0
\end{aligned}
\end{equation*}
where we recall $\theta= 2\,\lambda - \lambda\,\kappa\, -2\,\sigma^{2}>0$ and $a = \gamma\, L_{\mathfrak{m}}^{2} = \kappa\,(\lambda + 2\,\sigma^{2}\,\kappa)\, L_{\mathfrak{m}}^{2}$. For sufficiently small $\kappa$, we have $a-\theta <0$, and  in \eqref{contraction} we have $\Theta := \theta-a$. 

To prove uniqueness of the invariant measure, it is sufficient to consider to dynamics governed by \eqref{CBO kappa stat} but starting from two invariant measures $\rho^{\alpha}_{*,1}, \rho^{\alpha}_{*,2}$. Subsequently, the contraction result in  \eqref{contraction} yields 
\begin{equation*}
    \mathbb{W}_{2}(\rho^{\alpha}_{*,1}, \rho^{\alpha}_{*,2})^{2} \leq e^{-\Theta \, t}\,\mathbb{W}_{2}(\rho^{\alpha}_{*,1}, \rho^{\alpha}_{*,2})^{2} \quad \xrightarrow[t\to \infty]{} 0
\end{equation*}
which leads to a contradiction unless $\rho^{\alpha}_{*,1} = \rho^{\alpha}_{*,2}$, hence uniqueness.
\end{proof}

\begin{remark}\label{rmk: configuration 3}
We shall discuss the precise meaning of smallness of the parameters ensuring the validity of Theorem \ref{thm:long time}. 
As can be seen from its proof, the parameters need to verify the following two inequalities
\begin{equation}\label{necessary for contraction}
\begin{aligned}
    & \theta > a \; \Leftrightarrow \; 2\,\lambda - \lambda\,\kappa\, -2\,\sigma^{2} > \kappa\,(\lambda + 2\,\sigma^{2}\,\kappa)\, L_{\mathfrak{m}}^{2}\\
    \text{and } \quad & \theta >0 \; \Leftrightarrow \; 2\,\lambda - \lambda\,\kappa\, -2\,\sigma^{2} > 0
\end{aligned}
\end{equation}
The right-hand side of the first  inequality is $\mathcal{O}(\kappa)$ which means that it can be made arbitrarily small, hence ensuring the validity of the inequality. In order to obtain an example of configuration ensuring the latter inequality, let us choose $\lambda > 2\,\sigma^{2}$, and recall $\kappa \in (0,1)$. Therefore we have $\theta= 2\,\lambda - \lambda\,\kappa\, -2\,\sigma^{2}>0$. So we need $\kappa\,(\lambda + 2\,\sigma^{2}\,\kappa)\, L_{\mathfrak{m}}^{2} < 2\,\lambda - \lambda\,\kappa\, -2\,\sigma^{2}$. A \underline{sufficient condition} for Theorem \ref{thm:long time} to hold is when
\begin{equation*}
    \kappa\,(\lambda + 2\,\sigma^{2}\,\kappa)\, L_{\mathfrak{m}}^{2} < 2\lambda\,\kappa\, L_{\mathfrak{m}}^{2} < 2\,\lambda - \lambda\,\kappa\, -2\,\sigma^{2},
\end{equation*}
that is $ 2\lambda\,\kappa\, L_{\mathfrak{m}}^{2} + \lambda\,\kappa = \kappa\,\lambda (2L_{\mathfrak{m}}^{2} + 1) < 2\lambda - 2\sigma^{2}$, thus it suffices to have $\kappa < \frac{2\lambda - 2\sigma^{2}}{\lambda (2L_{\mathfrak{m}}^{2} + 1)}$.

Conclusion: Taking into account Remark \ref{rmk: configuration} and Remark \ref{rmk: configuation 2}, an example of sufficient conditions guaranteeing existence, uniqueness, and $\mathbb{W}_{2}$--exponential contraction is obtained when
\begin{equation*}
    \delta=1, \quad \lambda> \max\{16\,\sigma^{2}\;,\; 4\} \quad \text{ and } \quad 0<\kappa < \min\left\{\frac{2\lambda - 2\sigma^{2}}{\lambda (2L_{\mathfrak{m}}^{2} + 1)}\; , \; \frac{1}{2(1+b_2)}\right\}
\end{equation*}
where we recall $b_2$ is defined in \eqref{constant b1 b2}, and $L_{\mathfrak{m}}$ is in Lemma \ref{lem: stab}.
\end{remark}

\subsection{Uniform second moment bound}\label{subsec: unif}

The next result guarantees that the invariant measure has a uniform-in-$\alpha$ finite second moment. This will be important when studying the asymptotic $\alpha\to \infty$.

\begin{theorem}\label{thm: inv unif bound}
    In the situation of Theorem \ref{thm:long time}, the unique invariant probability measure $\rho^{\alpha}_{*}$ to the rescaled CBO \eqref{CBO kappa} satisfies 
    \begin{equation}\label{eq:sec}
     \rho_*^\alpha(|\cdot|^2)\leq \rho_0(|\cdot|^2)+C_2\,,
    \end{equation}
    where $C_2$ comes from Theorem \ref{thm: finite moments}
\end{theorem}

\begin{proof}
Recall $\rho_{t}^{\alpha} := \operatorname{Law}(\overline{X}_t)$ the law of \eqref{CBO kappa}. According to Theorem \ref{thm:long time}, one has
\begin{equation*}
    \mathbb{W}_{2}(\rho_{t}^{\alpha}, \rho_{*}^{\alpha})^{2} \; \leq \; \mathbb{W}_{2}(\rho_{0}^{\alpha}, \rho_{*}^{\alpha})^{2}\; e^{- \theta \, t} \quad \forall\, t\geq 0,\, \rho_{0}^{\alpha} \in \mathscr{P}_{2}(\mathbb{R}^{d})
\end{equation*}
where $\theta$ is a positive constant. 
This implies in particular: for any $|\varphi(x)|\leq C(1+|x|^2)$ it holds that
\begin{equation}\label{eq:contraction}
\lim_{t\to \infty} \int_{\RR^d}\varphi(x) \,  \dd\rho_t^\alpha (x)=\int_{\RR^d}\varphi(x) \, \dd\rho_*^\alpha (x)\,.
\end{equation}

Subsequently, for any fixed $r>0$, we can construct a $\mc{C}_c^\infty(\RR^d)$ bump function $\phi_r$ such that it has support on the closed ball $B_{2r}(0)$ satisfying
$|\phi_r(x)|\leq 1$ and is equal to 1 in $B_r(0)$. Then one can use \eqref{eq:contraction} with $\varphi(x) = \phi_{r}(x)|x|^{2}$, and obtain
\begin{equation*}
 \lim_{t\to \infty} \int_{\RR^d}\phi_r(x)|x|^2  \, \dd\rho_t^\alpha (x) = \int_{\RR^d}\phi_r(x)|x|^2  \, \dd\rho_*^\alpha (x)\,\mbox{ for any }r>0\,.
\end{equation*}
Recalling the uniform-in-time estimate in \eqref{eq:uni}, we have
\begin{equation*}
    \int_{\RR^d}\phi_r(x)|x|^2  \, \dd\rho_*^\alpha (x) = \lim_{t\to \infty} \int_{\RR^d}\phi_r(x)|x|^2  \, \dd\rho_t^\alpha (x)\leq \limsup_{t\to \infty} \int_{\RR^d}|x|^2 \,  \dd\rho_t^\alpha (x)\leq \rho_0(|\cdot|^2)+C_2\,.
\end{equation*}
Furthermore, one can use Fatou's lemma and obtain
\begin{equation*}
    \rho_*^\alpha(|\cdot|^2)=\int_{\RR^d}|x|^2 \,  \dd\rho_*^\alpha (x)\leq \liminf_{r\to\infty}\int_{\RR^d}\phi_r(x)|x|^2 \,  \dd\rho_*^\alpha (x)\leq \rho_0(|\cdot|^2)+C_2\,.
\end{equation*}
\end{proof}

\subsection{Harnack's inequality}\label{sec: Harnack}\label{subsec: Harnack}

The result in \eqref{eq:sec} together with \eqref{lemeq2}  ensure that $\mathfrak{m}_{\alpha}(\cdot)$ is well-defined (finite) when evaluated in the invariant probability measure $\rho_*^\alpha$.  Namely, it holds that
\begin{equation}\label{eq:sec m}
 |\mathfrak{m}_{\alpha}(\rho_*^\alpha)|\leq (b_1+b_2\rho_*^\alpha(|\cdot|^2))^{\frac{1}{2}}\leq (b_1+b_2(\rho_0(|\cdot|^2)+C_2))^{\frac{1}{2}}<\infty\,. 
\end{equation}
Therefore, we can introduce the following matrix-valued and vector-valued functions defined for $(x,\mu) \in \RR^d \times \mathscr{P}_{2}(\RR^d)$ by
\begin{equation}\label{eq: A b}
\begin{aligned}
	&A:=(a^{ij})_{i,j}, \quad \text{ such that } \quad a^{ij}(x,\mu):=\frac{\sigma^2}{2}(\delta+|(x-\kappa\,\mathfrak{m}_{\alpha}(\mu))_i|)^2\delta_{ij} \\
    \mbox{ and } & b :=(b^i)_{i},  \quad \text{ such that } \quad b^i(x,\mu):=- \lambda(x-\kappa\,\mathfrak{m}_{\alpha}(\mu))_i
\end{aligned}
\end{equation}
$\delta_{ij}$ being the Kronecker symbol whose value is $0$ whenever $i\neq j$ (not to be confused with the positive constant $\delta$). 
Moreover, the matrix $A$ is positive--definite because it is a diagonal matrix whose diagonal entries satisfy $\frac{\sigma^2}{2}(\delta+|(x-\kappa\,\mathfrak{m}_{\alpha}(\mu))_i|)^2 \geq \frac{(\sigma\delta)^{2}}{2}>0$. 
In particular, when evaluated in the invariant probability measure $\rho_*^\alpha$, one obtains smooth  coefficients (in $\mc{C}^\infty (\RR^d)$)
\begin{align*}
	&A_{*}:=(a^{ij}_{*})_{i,j}, \quad \text{ such that } \quad a^{ij}_{*}(x) := a^{ij}(x,\rho_*^\alpha)\\
    \mbox{ and } & b_{*} :=(b^i_{*})_{i}, \quad \text{ such that } \quad b^i_{*}(x) := b^i(x,\rho_*^\alpha).
\end{align*} 
This identification is possible since the dependence of $A$ and $b$ on the measure is non-local -- that is, it involves a quantity of the form $\int_{\RR^d} K(y) \dd\mu(y)$ which is moreover constant in $x$ (since the kernel $K$ does not involve $x$)--, rather than a pointwise (local) evaluation of (the density of) the measure $\mu$. In other words, regularity of (the density of) the measure $\mu$ does not influence regularity of the coefficients $A$ and $b$.

More generally, given a measure $\nu \in \mathscr{P}_{2}(\RR^d)$, the quantity $|\mathfrak{m}_{\alpha}(\nu)|$ is well-defined (finite) thanks to \eqref{lemeq2}, and we can define the following functions on $\RR^{d}$
\begin{align*}
	&A_{\nu}:=(a^{ij}_{\nu})_{i,j}, \quad \text{ such that } \quad a^{ij}_{\nu}(x) := a^{ij}(x,\nu)\\
    \mbox{ and } & b_{\nu} :=(b^i_{\nu})_{i}, \quad \text{ such that } \quad b^i_{\nu}(x) := b^i(x,\nu).
\end{align*}
The diffusion (linear) operator whose coefficients are $A_{\nu}$ and $b_{\nu}$ is given by
\begin{equation*}
    L_{A,b,\nu}\psi(x) =  \text{trace}\big(A_{\nu}(x)D^2\psi(x)\big) + 
    b_{\nu}(x) \cdot \nabla\psi(x) \quad \forall \; \psi \in \mc{C}_c^\infty (\RR^d),
\end{equation*}
Its adjoint is denoted by $L_{A,b,\nu}^{*}$, and a measure $\mu$ is solution to $L_{A,b,\nu}^{*}\mu=0$ if it is a solution in the distributional sense that is
\begin{equation*}
    \int_{\RR^d} L_{A,b,\nu}\psi(x)\; \dd \mu(x) =0 \quad \forall \; \psi \in \mc{C}_c^\infty (\RR^d).
\end{equation*}
The diffusion (linear) operator whose coefficients are $A_{*}$ and $b_{*}$ will be denoted by $L_{A,b,*}$, and its adjoint is $L_{A,b,*}^{*}$. It should be noted that the choice of the test functions being $\psi \in \mc{C}_c^\infty (\RR^d)$ can be relaxed to include a larger space of functions but which we shall not do here for the sake of simplicity. 

When we allow the coefficients $A$ and $b$ to be also functions of the measure argument, the resulting diffusion operator is non-linear. Analogously, we say that $\mu$ is solution to $L_{A,b,\mu}^{*}\mu=0$ if it is a solution in the distributional sense that is
\begin{equation*}
    \int_{\RR^d} L_{A,b,\mu}\psi(x)\; \dd \mu(x) =0 \quad \forall \; \psi \in \mc{C}_c^\infty (\RR^d).
\end{equation*}

The following result is well-known and can be found for example in \cite{bogachev2022fokker}. 

\begin{theorem}\label{thm general}
    Let $\nu \in \mathscr{P}_2(\RR^d)$ be fixed. Let $A_{\nu}$ and $b_{\nu}$ be as defined above. Let $\mu$ be a locally finite Borel measure satisfying $L_{A,b,\nu}^{*}\mu=0$. Then $\mu$ has a density denoted by $\varrho$ in $W_{\text{loc}}^{1,p}(\RR^d)$ with $(p>d)$ that is locally H\"{o}lder continuous. If moreover $\mu$ is nonnegative, then the continuous version of the density $\varrho$ of $\mu$ satisfies Harnack's inequality: for every compact set $K$ contained in a connected set $U$, it holds 
    \begin{equation*}
        \sup_{K}\varrho \;\leq\; C\, \inf_{K}\varrho
    \end{equation*}
    where $C$ depends only on $\|a^{ij}_{\nu}\|_{W^{1,p}(U)}$, $\|b^{i}_{\nu}\|_{L^{p}(U)}$, $\inf_{U}\det(A_{\nu})$, and $K$.
\end{theorem}

\begin{proof}
First observe that the coefficients $A_{\nu}$ and $b_{\nu}$ are smooth, and $A_{\nu}$ is positive--definite thanks to  $\delta>0$ in the diffusion. Then,  existence and regularity of the density is \cite[Corollary 1.6.7]{bogachev2022fokker}. Harnack's inequality is \cite[Corollary 1.7.2]{bogachev2022fokker}.
\end{proof}

\begin{corollary}\label{cor: Harnack}
The unique invariant probability measure $\rho_*^\alpha$ to the rescaled CBO \eqref{CBO kappa} has a density (still denoted by $\rho_*^\alpha$) in $W_{\text{loc}}^{1,p}(\RR^d)$ with $(p>d)$ that is locally H\"{o}lder continuous. Moreover, its continuous version satisfies Harnack's inequality: for any ball $B_r(x_0):= \big\{x\in\RR^{d}: |x-x_0|\leq r\big\}$ with $x_0\in \RR^d$, one has
\begin{equation}\label{eq:harnack}
    \sup_{x\in B_r(x_0)}\rho_*^\alpha(x) \leq C \inf_{x\in B_r(x_0)}\rho_*^\alpha(x)
\end{equation}
where $C$ depends only on $\|a^{ij}\|_{W^{1,p}(B_{2r}(x_0))}$, $\|b^{i}\|_{L^{p}(B_{2r}(x_0))}$ and $\inf_{B_{2r}(x_0)}\det(A)$.
\end{corollary}

\begin{proof}
Existence and uniqueness of an invariant probability measure $\rho_*^\alpha$ is obtained in \cite[Proposition 3.4 \& Corollary 3.8]{huang2024self}, that is, $\rho_*^\alpha$ is the unique solution to the (non-linear) equation $L_{A,b,\mu}^{*}\mu=0$ in the class of probability measures. Moreover, \eqref{eq:sec} ensures that $\rho_*^\alpha$ has a finite second moment. Therefore, we can consider the diffusion (linear) operator $L_{A,b,*}$ and we are in the setting of Theorem \ref{thm general} which gives the desired result when we choose the solution to $L_{A,b,*}^{*}\mu=0$ within the probability measures. Indeed in this case, the unique probability measure $\mu$ satisfying $L_{A,b,*}^{*}\mu=0$ is $\rho_*^\alpha$. Finally, the sets $U$ and $K$ in Theorem \ref{thm general} are chosen as the closed balls $B_{2r}(x_0)$ and $B_r(x_0)$ respectively. 
\end{proof}

\begin{lemma}\label{lem:mass}
Assume $f(\cdot)$ satisfies Assumption \ref{assum1}, and let $\rho_*^\alpha$ be the unique invariant probability measure of the rescaled CBO \eqref{CBO kappa}. Let $x_{\circ} \in\RR^d$ be an arbitrarily fixed point, then for any $\varepsilon>0$ there exists some $C_\varepsilon>0$ depending only on $\rho_0(|\cdot|^2),\, C_2, \, x_{\circ}$ and $\varepsilon$ such that
\begin{equation}
    \rho_*^\alpha(B_{\varepsilon}(x_{\circ}))\geq C_\varepsilon,
\end{equation}
with the constant $C_2$ being the one from Theorem \ref{thm: finite moments}.
\end{lemma}

\begin{proof}
For any $\varepsilon>0$, let us consider some $R>\varepsilon$ and use Harnack's inequality  \eqref{eq:harnack}
\begin{align*}
    \sup\limits_{B_R(x_{\circ})}\rho_*^\alpha & \leq C_{B_{2R}(x_{\circ})}\inf\limits_{B_R(x_{\circ})}\rho_*^\alpha \\
    & \leq  C_{B_{2R}(x_{\circ})}\inf\limits_{B_{\varepsilon}(x_{\circ})}\rho_*^\alpha  %\\
     \leq  C_{B_{2R}(x_{\circ})}\frac{1}{|B_{\varepsilon}(x_{\circ})|}\rho_*^\alpha(B_{\varepsilon}(x_{\circ}))\,.
\end{align*}
Here $C_{B_{2R}(x_{\circ})}$ is a positive constant depending on  
$\|a^{ij}\|_{W^{1,p}(B_{2R}(x_0))}$, $\|b^{i}\|_{L^{p}(B_{2R}(x_0))}$ and $\inf_{B_{2R}(x_0)}\det(A)$. Using \eqref{eq:sec m}, it is not difficult to see that these quantities depend only on $B_{2R}(x_{\circ})$ and on the constants $\delta, \kappa, \lambda, \sigma, R, b_1, b_2, C_2$ which are independent from $\alpha$.

On the other hand, it is easy to check that
\begin{align*}
    \sup\limits_{B_R(x_{\circ})}\rho_*^\alpha& \geq \frac{1}{|B_R(x_{\circ})|}\rho_*^\alpha(B_R(x_{\circ})) \\ 
    & \geq \frac{1}{|B_R(x_{\circ})|}(1-\frac{\rho_*^\alpha(|\cdot|^2)}{R^2})  \\ 
    & \geq  \frac{1}{|B_R(x_{\circ})|}(1-\frac{\rho_0(|\cdot|^2)+C_2}{R^2})
\end{align*}
where we have used Markov's inequality in the second inequality, and \eqref{eq:sec} in the last inequality. We can now take $R$ large enough such that $1-\frac{\rho_0(|\cdot|^2)+C_2}{R^2} \geq \frac{1}{2}$, and obtain 
\begin{equation*}
    \sup\limits_{B_R(x_{\circ})}\rho_*^\alpha\geq \frac{1}{2|B_R(x_{\circ})|}\,,
\end{equation*}
which together with the previous inequalities imply
\begin{equation*}
    \rho_*^\alpha(B_{\varepsilon}(x_{\circ}))\geq  \frac{|B_{\varepsilon}(x_{\circ})|}{2 \, C_{B_{2R}(x_{\circ})}|B_R(x_{\circ})|} =: C_\varepsilon >0.
\end{equation*}
\end{proof}

\section{Laplace's principle and global convergence}\label{sec: convergence}

We are now ready to state and prove our main results concerning the limit $\alpha \to \infty$. We shall first treat the case where the function $f(\cdot)$ to be minimized has a unique  minimizer in \S \ref{subsec: single min}. Then, and although CBO methods are primarily designed for objective functions with a single minimizer, we are able to treat in \S \ref{subsec: multi min} the more general situation where $f(\cdot)$ has multiple minimizing regions. 

Let us recall that  our results hold without any prior assumptions on the initial distribution $\rho_0\in\mathscr{P}_{4}(\mathbb{R}^{d})$ of \eqref{CBO kappa}. In particular, we do not assume that a global minimizer of $f(\cdot)$ is contained within the support of $\rho_0$. 
This absence of conditions on the initial distribution is precisely what ensures the \textit{global} nature of the convergence. 

\subsection{Single--minimizer}\label{subsec: single min}

\begin{assum}\label{assum2}
$f(\cdot)$ has a unique global minimizer $x_*$.
\end{assum}

The following observation is useful for the proof of the subsequent theorem. 

\begin{remark}
Thanks to the continuity and coercivity of $f(\cdot)$ in Assumption \ref{assum1}, for all arbitrary small $\varepsilon>0$ it holds that 
\begin{equation}\label{eq:min isolated}
    \inf\{f(x): |x-x_*|\geq \varepsilon\} >f(x_*)=:\underline f\,.
\end{equation}
 See also Remark \ref{rmk:flat} for further references.  
\end{remark}

The following result is our first Laplace's principle in the case of a unique minimizer.

\begin{theorem}\label{thm:conver}
    Assume $f(\cdot)$ satisfies Assumptions \ref{assum1} and \ref{assum2}, and let $\rho_*^\alpha$ be the unique invariant probability measure of the rescaled CBO \eqref{CBO kappa}. Define the probability measure
    \begin{equation}\label{eq:eta}
        \eta_*^\alpha(A):=\frac{\int_{A}e^{-\alpha f(x)}\rho_*^\alpha(\dd x)}{\int_{\RR^d}e^{-\alpha f(y)}\rho_*^\alpha(\dd y)},\quad A\in \mc{B}(\RR^d).
    \end{equation}
    Then for any $\varepsilon>0$, it holds 
    \begin{equation*}
        \eta_*^\alpha(\{x \in \RR^d :|x-x_*|\geq \varepsilon\})\longrightarrow 0  \quad \mbox{ as }  \quad \alpha \to \infty.
    \end{equation*}
\end{theorem}
\begin{proof}
Let $\varepsilon>0$ be fixed, and let us define
\begin{equation*}
    \overline f_\varepsilon:=\sup\{f(x): |x-x_*|\leq \varepsilon\} \mbox{ and } \underline f_\varepsilon:=\inf\{f(x): |x-x_*|\geq \varepsilon\}. 
\end{equation*}
Since $f(\cdot)$ is continuous at $x_*$ we know that for any $\varepsilon>0$, there exists some $\varepsilon'>0$ such that
\begin{equation}\label{diff f eps}
    \overline f_{\varepsilon'} \leq \frac{1}{2} (\underline f_\varepsilon + \underline f), \mbox{ that is }\quad
    2(\overline f_{\varepsilon'}-\underline f) \leq (\underline f_\varepsilon-\underline f)\,.
\end{equation}

Then we compute
\begin{align*}
    \eta_*^\alpha (\{x \in \RR^d :|x-x_*|\geq \varepsilon\} )
    & =
   \frac{\int_{|x-x_*|\geq \varepsilon}e^{-\alpha f(x)} \,\rho_*^\alpha(\dd x)}{\int_{\RR^d}e^{-\alpha f(x)} \, \rho_*^\alpha(\dd x)}\\
   & \leq \frac{\int_{|x-x_*|\geq \varepsilon}e^{-\alpha (f(x)-\overline f_{\varepsilon'})} \, \rho_*^\alpha(\dd x)}{\int_{|x-x_*|\leq \varepsilon'}e^{-\alpha (f(x)-\overline f_{\varepsilon'})} \, \rho_*^\alpha(\dd x)}.
\end{align*} 

\textbullet\, In $\{x\in \RR^d : |x-x_*|\geq \varepsilon\}$, we have $f(x) \geq \underline f_\varepsilon$, hence $e^{-\alpha (f(x)-\overline f_{\varepsilon'})} \leq e^{-\alpha (\underline f_\varepsilon -\overline f_{\varepsilon'})}$. Moreover, we have $\underline f_\varepsilon -\overline f_{\varepsilon'} = (\underline f_\varepsilon - \underline f) - (\overline f_{\varepsilon'} - \underline f) \geq (\overline f_{\varepsilon'} - \underline f)>0$ using \eqref{diff f eps}. Therefore, we have 
\begin{align*}
    e^{-\alpha (f(x)-\overline f_{\varepsilon'})} \leq e^{-\alpha (\underline f_\varepsilon -\overline f_{\varepsilon'})} \leq e^{-\alpha (\overline f_{\varepsilon'} - \underline f)} 
\end{align*}
and $\rho_*^\alpha$ being a probability measure, we have 
\begin{align*}
    \int_{|x-x_*|\geq \varepsilon}e^{-\alpha (f(x)-\overline f_{\varepsilon'})} \, \rho_*^\alpha(\dd x)  \leq e^{-\alpha (\overline f_{\varepsilon'} - \underline f)} \; \to 0 \; \mbox{ when } \;\alpha \to \infty.
\end{align*}

\textbullet\, In $\{x\in \RR^d : |x-x_*|\leq \varepsilon'\}$, we have $f(x) \leq \overline f_{\varepsilon'}$, hence $e^{-\alpha (f(x)-\overline f_{\varepsilon'})} \geq 1$, and
\begin{align*}
    \int_{|x-x_*|\leq \varepsilon'}e^{-\alpha (f(x)-\overline f_{\varepsilon'})} \, \rho_*^\alpha(\dd x)  \geq \rho_*^\alpha(B_{\varepsilon'}(x_*)) \geq  C_{\varepsilon'}>0
\end{align*}
where in the last inequality $C_{\varepsilon'}$ comes from Lemma \ref{lem:mass}, which is independent of $\alpha$.

Finally, one gets
\begin{align*}
    \eta_*^\alpha(\{x \in \RR^d :|x-x_*|\geq \varepsilon\} )
   &\leq \frac{e^{-\alpha (\overline f_{\varepsilon'} - \underline f)}}{C_{\varepsilon'}}\,\; \to 0 \; \mbox{ when } \;\alpha \to \infty.
\end{align*}

\end{proof}

We can now readily state and prove the main result when the function $f(\cdot)$ to be minimized is assumed to have a unique global minimizer. 

\begin{corollary}\label{cor: conv to min}
Assume $f(\cdot)$ satisfies Assumption \ref{assum1}. 
Let $\OX_.$ be the solution to the rescaled CBO \eqref{CBO kappa}, and let $\rho_*^\alpha$ be its unique invariant probability measure. Then 
\begin{equation*}
    \lim_{t\to \infty} \; \EE\left[\,\OX_t\,\right]=\int_{\RR^d}x\rho_*^\alpha(\dd x) =\kappa\, \mathfrak{m}_{\alpha}(\rho_*^\alpha)
    =\kappa\,\lim\limits_{t\to \infty} \mathfrak{m}_{\alpha}(\rho^{\alpha}_{t}).
\end{equation*}
If moreover $f(\cdot)$ satisfies Assumption \ref{assum2}, then 
\begin{equation*}
    \lim_{\alpha \to \infty} \lim_{t\to \infty} \; \kappa^{-1}\,\EE\left[\,\OX_t\,\right]
    =\lim_{\alpha \to \infty} \lim\limits_{t\to \infty} \mathfrak{m}_{\alpha}(\rho^{\alpha}_{t})
    = \,\lim_{\alpha \to \infty} \mathfrak{m}_{\alpha}(\rho_*^\alpha)
= \, x_*.
\end{equation*}
\end{corollary}

\begin{proof}
Let us first recall that existence and uniqueness of the invariant measure $\rho^{\alpha}_{*}$ are given by Theorem \ref{thm: existence} and Theorem \ref{thm:long time} (see also Remark \ref{rmk: configuration 3}). \\
By definition, a probability measure $\rho_*^\alpha$ is \textit{invariant} for $\OX_{\cdot}$ satisfying \eqref{CBO kappa} if and only if it is a fixed point for the adjoint of its (nonlinear) transition semigroup $\{T_{t}\}_{t\geq 0}$, that is $T^{*}_{t}\rho_*^\alpha = \rho_*^\alpha$ for all $t> 0$; the adjoint operator $T^{*}_{t}$ being defined on the space of probability measures as $T^{*}_{t}\nu := \text{Law}(\OX_{t})$ when $\nu=\text{Law}(\OX_{0})$. 
Recalling the dynamics \eqref{CBO kappa}, we have
\begin{equation}\label{eq:to take expectation}
\begin{aligned}
    \OX_t 
    & = \OX_0-\lambda \int_0^t(\OX_s-\kappa\,\mathfrak{m}_{\alpha}(\rho^{\alpha}_s)) \, \ds + \sigma \int_0^t \left(\delta\,\mathds{I}_{d} + D(\OX_t - \kappa\, \mathfrak{m}_{\alpha}(\rho^{\alpha}_s)) \right) \, \rd B_s \\
    & =  \OX_0-\lambda \int_0^t(\OX_s-\kappa\,\mathfrak{m}_{\alpha}(\rho^{\alpha}_s)) \, \ds +  \text{martingale (thanks to \eqref{eq:uni} and \eqref{eq:sec m})}
\end{aligned}
\end{equation}
where  $\rho^{\alpha}_s:= \text{Law}(\OX_{s})$. 
In particular, when choosing as initial distribution $\text{Law}(\OX_{0}) = \rho_*^\alpha$, then it holds
$\text{Law}(\OX_{t}) = \rho_*^\alpha$ for all $t> 0$ since the latter probability is invariant. 

Let us now consider the (stationary) process \eqref{eq:to take expectation} with initial distribution $\OX_{0}\sim \rho_*^\alpha$. It is well-posed thanks to \eqref{eq:sec m}.  Then taking the expectation\footnote{Note that this is why the value of $\delta>0$ does not influence the rest of the proof.} in \eqref{eq:to take expectation} and bearing in mind the invariance of  $\rho_*^\alpha$, one has
\begin{align*}
    \int_{\RR^d}x\rho_*^\alpha (\dd x)=\int_{\RR^d}x\rho_*^\alpha(\dd x)-\lambda t \left(\int_{\RR^d}x\rho_*^\alpha(\dd x)-\kappa\,\mathfrak{m}_{\alpha}(\rho_*^\alpha)\right)\,
\end{align*}
which further implies that
\begin{equation*}
	\int_{\RR^d}x\rho_*^\alpha(\dd x) = \kappa \,\mathfrak{m}_{\alpha}(\rho_*^\alpha).
\end{equation*}
We can now consider (using again the same notation) the process \eqref{eq:to take expectation} but with an arbitrary initial distribution $\OX_{0}\sim \rho_0 \in\mathscr{P}_{4}(\RR^d)$. Then, the result in \eqref{eq:contraction} yields
\begin{equation*}
	\lim_{t\to \infty} \; \EE\left[\,\OX_t\,\right] = \lim_{t\to \infty} \;
    \int_{\RR^d}x\rho_t^\alpha (\dd x) =
    \int_{\RR^d}x\rho_*^\alpha (\dd x)
\end{equation*}
which together with the previous equality yields $\lim_{t\to \infty} \; \EE\left[\,\OX_t\,\right]=\kappa\, \mathfrak{m}_{\alpha}(\rho_*^\alpha)$. Finally, using the conclusion of Corollary \ref{cor: long time of m}, one obtains the first statement.% of the corollary.

Next, to prove the limit when $\alpha\to \infty$, we introduce $Y_*^\alpha$ as a random variable with probability distribution $\eta_*^\alpha$ defined in \eqref{eq:eta}. Then we have
\begin{align*}
  |\mathfrak{m}_{\alpha}(\rho_*^\alpha)-x_*|&\leq \EE[|Y_*^\alpha-x_*|]=\int_{\RR^d}|x-x_*|\eta_*^\alpha(\dd x)\\
  &=\int_{|x-x_*|\leq \varepsilon}|x-x_*|\eta_*^\alpha(\dd x)+\int_{|x-x_*|\geq \varepsilon}|x-x_*|\eta_*^\alpha(\dd x)\\
  &\leq \varepsilon+ \bigg(\eta_*^\alpha(\{x:|x-x_*|\geq \varepsilon\}) \bigg)^{\frac{1}{2}} \left(\int_{|x-x_*|\geq \varepsilon}|x-x_*|^2\eta_*^\alpha(\dd x) \right)^{\frac{1}{2}}\\
  &\leq \varepsilon + \bigg(\eta_*^\alpha(\{x:|x-x_*|\geq \varepsilon\}) \bigg)^{\frac{1}{2}} \bigg( 2 |x_*|^2 + 2\,\eta_*^\alpha(|\cdot|^2)\bigg)^{\frac{1}{2}}
\end{align*}
Recalling \eqref{lemeq2} and \eqref{eq:sec}, we have 
\begin{equation}\label{bound second moment eta}
  \eta_*^\alpha(|\cdot|^2)\leq b_1+b_2\rho_*^\alpha(|\cdot|^2)\leq b_1+b_2(\rho_0(|\cdot|^2)+C_2)
\end{equation}
and then 
\begin{equation*}
    |\mathfrak{m}_{\alpha}(\rho_*^\alpha)-x_*|\leq \varepsilon+\bigg(\eta_*^\alpha(\{x:|x-x_*|\geq \varepsilon\}) \bigg)^{\frac{1}{2}} \bigg(2 |x_*|^2 + 2b_1 + 2b_2(\rho_0(|\cdot|^2)+C_2)\bigg)^{\frac{1}{2}}.
\end{equation*}
Thus, we can take $\alpha\to\infty$ while using Theorem \ref{thm:conver}, and obtain
\begin{equation*}
   \lim_{\alpha \to \infty}\, |\mathfrak{m}_{\alpha}(\rho_*^\alpha)-x_*|\leq \varepsilon
\end{equation*}
which holds for any arbitrarily chosen $\varepsilon>0$. In particular, when letting $\varepsilon \to 0$, we obtain the desired result: $\; \mathfrak{m}_{\alpha}(\rho_*^\alpha)\to x_* $ when $ \alpha \to \infty.$ 
\end{proof}

\subsection{Multiple--minimizer} \label{subsec: multi min}

The strategy developed in the previous subsection can be further generalized to a set of minimizers that is neither a unique singleton, nor a collection of discrete points. More specifically, we shall drop the assumption on $f(\cdot)$ having a unique minimizer, and we consider the following.
\begin{assum}\label{assum4}
The function $f(\cdot)$ attains its minimal value on a set $\mc{M} = \cup_{i=1}^{m} \, \mc{O}_i$,  with $1\leq m<\infty $ and where every $\mc{O}_i$ is a compact and connected subset of $\RR^d$.
\end{assum}

On each $\mc{O}_i$, the function $f(\cdot)$ is constant equal to its minimal value. When $\mc{O}_i$ is not a singleton, the graph (landscape) of $f(\cdot)$ is flat, which is why these are sometimes called \textit{flat minimizers}. It is also worth mentioning that a particular example of Assumption \ref{assum4} is when $f(\cdot)$ has $m$ global minimizers $\{x_*^1,\dots,x_*^m\}=: \mc{M}$. In this case, $\mathcal{O}_{i} = \{x_{*}^{i}\}$.

\begin{remark}\label{rmk:flat isolated}
The function $f(\cdot)$ being continuous and coercive thanks to Assumption \ref{assum1}, for all arbitrary $\varepsilon>0$ it holds that
    \begin{equation}\label{eq:flat min isolated}
        \inf\{ \, f(x): x \in \mc{M}^c_\varepsilon \,\} >\underline f\,
    \end{equation}
    where 
    \begin{equation*}
        \mc{M}_\varepsilon := \cup_{i=1}^{m} \, \{x\in \RR^d : \mathrm{dist}(x,\mc{O}_i)  \leq \varepsilon\}
    \end{equation*}
    and $\mc{M}^c_\varepsilon = \RR^d \setminus \mc{M}_\varepsilon$ its complement.  Here, $\mathrm{dist}(x,\mc{O})$ is the Euclidean distance between a point $x\in \RR^d$ and a compact set $\mc{O}$, and which is equal to $0$ when $x\in \mc{O}$. 
\end{remark}

\begin{remark}\label{rmk:flat}
Property \eqref{eq:flat min isolated}, which in the case of a single minimizer is \eqref{eq:min isolated}, has already been discussed in \cite[equation (H)]{bardi2023eikonal}; see \cite[Remark 2]{bardi2023eikonal} for a counter-example, and a variation in \cite[Assumption (2.3)]{bardi2024long}. A qualitative interpretation can be found in \cite[Assumption (A3)]{huang2025nondyn} and it is also used in \cite[(A3)]{huang2025nonconv} where several examples of optimization problems satisfying this condition are provided. All of the works cited in this remark invoke this property in their respective treatments of global optimization, thereby emphasizing its natural place as a standard assumption.
\end{remark}

The first result we need in the sequel is the following tightness of the weighted probability measures. Recall that existence and uniqueness of the invariant measure $\rho_*^\alpha$ were the object of Theorem \ref{thm: existence} and Theorem \ref{thm:long time}  (see also Remark \ref{rmk: configuration 3}).

\begin{lemma}\label{lem conv}
    Let $\rho_*^\alpha$ be the unique invariant probability measure of the rescaled CBO \eqref{CBO kappa}, and let $\eta_*^\alpha$ be defined as in \eqref{eq:eta}.
    Then $\{\eta_*^\alpha \}_{\alpha\geq 1}$ is tight.
\end{lemma}

\begin{proof}
    It follows from \eqref{lemeq2} and \eqref{eq:sec} 
  \begin{equation*}
      \eta_*^\alpha(|\cdot|^2)\leq b_1 + b_2 \,\rho_*^\alpha(|\cdot|^2)\leq b_1 + b_2 \big(\rho_0(|\cdot|^2)+C_2\big) =: C_*\,.
  \end{equation*}
  Therefore, for any $\varepsilon$, there exists a compact set $K_\varepsilon:=\{x \in \RR^d \,:\,|x|^2\leq \frac{C_*}{\varepsilon}\}$ such that, by Markov’s inequality, it holds
  \begin{equation*}
      \eta_*^\alpha(K_\varepsilon^c)=\eta_*^\alpha\left(\left\{x \in \RR^d \,:\, |x|^2>\frac{C_*}{\varepsilon} \right\} \right)\leq \frac{\varepsilon \, \eta_*^\alpha(|\cdot|^2)}{C_*}\leq \varepsilon.
  \end{equation*}
\end{proof}

Thanks to the latter Lemma, one applies Prokhorov's theorem and concludes that there exist a subsequence $\{\eta_*^{\alpha_k} \}_{k\geq 1}$ and a probability measure $\eta_*\in \mathscr{P}(\RR^d)$ such that
$\eta_*^{\alpha_k}$ converges weakly to $\eta_*$ as $k \to \infty$.  In the following we still use the notation $\eta_*^\alpha$ to denote the subsequence. 

The next theorem can be viewed as  Laplace's principle in the case of multiple non-trivial minimizers.

\begin{theorem}\label{thm:conver3}
Let $f(\cdot)$ satisfy Assumptions \ref{assum1} and \ref{assum4}, and let $\rho_*^\alpha$ be the unique invariant probability measure of the rescaled CBO \eqref{CBO kappa}. Let $\eta_*^\alpha$ be defined as in \eqref{eq:eta}. 
Then for all $\varepsilon>0$, we have
\begin{equation*}
    \eta_*^\alpha (\mc{M}_{\varepsilon} )
   \; \to 1 \; \mbox{ when } \;\alpha \to \infty.
\end{equation*}
Moreover, when $\alpha \to \infty$, any weak limit $\eta_*$ of the sequence $\{\eta_*^\alpha \}_{\alpha\geq 1}$ concentrates on $\mc{M}$.
\end{theorem}

\begin{proof} 
For any $\varepsilon>0$, let us define
\begin{equation*}
    \overline f_\varepsilon:=\sup\{f(x): x \in \mc{M}_\varepsilon\} \mbox{ and } \underline f_\varepsilon:=\inf\{f(x): 
    x \in \mc{M}_\varepsilon^c \} 
\end{equation*}
where $\mc{M}_\varepsilon$ is as defined in Remark \ref{rmk:flat isolated}. 
Since $f$ is continuous, we know that for any $\varepsilon>0$, there exists some $\varepsilon'>0$ such that
\begin{equation}\label{diff f eps 3}
    \overline f_{\varepsilon'} \leq \frac{1}{2} (\underline f_\varepsilon + \underline f), \mbox{ that is }\quad
    2(\overline f_{\varepsilon'}-\underline f) \leq (\underline f_\varepsilon-\underline f)\,.
\end{equation}

Then we compute
\begin{align*}
    \eta_*^\alpha (\mc{M}_\varepsilon^c )
    & =
   \frac{\int_{\mc{M}_\varepsilon^c}e^{-\alpha f(x)} \, \rho_*^\alpha(\dd x)}{\int_{\RR^d}e^{-\alpha f(x)} \, \rho_*^\alpha(\dd x)}\; \leq \frac{\int_{\mc{M}_\varepsilon^c} e^{-\alpha (f(x)-\overline f_{\varepsilon'})} \, \rho_*^\alpha(\dd x)}{\int_{\mc{M}_{\varepsilon'}} e^{-\alpha (f(x)-\overline f_{\varepsilon'})} \, \rho_*^\alpha(\dd x)}.
\end{align*} 

\textbullet\, In the set $\mc{M}_\varepsilon^c $, we have $f(x) \geq \underline f_\varepsilon$, hence $e^{-\alpha (f(x)-\overline f_{\varepsilon'})} \leq e^{-\alpha (\underline f_\varepsilon -\overline f_{\varepsilon'})}$. Moreover using \eqref{diff f eps 3}, it holds
\begin{equation*}
    \underline f_\varepsilon -\overline f_{\varepsilon'} = (\underline f_\varepsilon - \underline f) - (\overline f_{\varepsilon'} - \underline f) \geq (\overline f_{\varepsilon'} - \underline f)>0.
\end{equation*}
Therefore, we have 
\begin{align*}
    e^{-\alpha (f(x)-\overline f_{\varepsilon'})} \leq e^{-\alpha (\underline f_\varepsilon -\overline f_{\varepsilon'})} \leq e^{-\alpha (\overline f_{\varepsilon'} - \underline f)} 
\end{align*}
and $\rho_*^\alpha$ being a probability measure, we have 
\begin{align*}
    \int_{\mc{M}_\varepsilon^c}e^{-\alpha (f(x)-\overline f_{\varepsilon'})} \, \rho_*^\alpha(\dd x)  \leq e^{-\alpha (\overline f_{\varepsilon'} - \underline f)} \; \to 0 \; \mbox{ when } \;\alpha \to \infty.
\end{align*}

\textbullet\, In the set $\mc{M}_{\varepsilon'}$, we have $f(x) \leq \overline f_{\varepsilon'}$, hence $e^{-\alpha (f(x)-\overline f_{\varepsilon'})} \geq 1$, and
\begin{align*}
    \int_{\mc{M}_{\varepsilon'}}e^{-\alpha (f(x)-\overline f_{\varepsilon'})} \, \rho_*^\alpha(\dd x)  \geq \rho_*^\alpha(\mc{M}_{\varepsilon'}) \geq  C_{\varepsilon'}>0
\end{align*}
where in the last inequality $C_{\varepsilon'}$ comes from Lemma \ref{lem:mass}, which is independent of $\alpha$.

Finally, one gets
\begin{equation}\label{lim thm}
    \eta_*^\alpha (\mc{M}_\varepsilon^c )
   \leq \frac{e^{-\alpha (\overline f_{\varepsilon'} - \underline f)}}{C_{\varepsilon'}}\,\; \to 0 \; \mbox{ when } \;\alpha \to \infty.
\end{equation}
It follows from the fact that $\eta_*^{\alpha}$ converges weakly to $\eta_*$  as $\alpha \to \infty$ (Lemma \ref{lem conv}), that
\begin{equation*}
 \eta_*(A)\leq \liminf_{\alpha \to \infty }\eta_*^\alpha(A)
\end{equation*}
for any open set $A\subset \RR^d$. It is obvious that 
$\mc{M}_\varepsilon^c $ is an open set. Thus we have $\eta_*(\mc{M}_\varepsilon^c ) \leq 0\,$, which implies that $\eta_*(\mc{M}_\varepsilon) =1\,$. 
Since $\varepsilon$ can be made arbitrarily small, we have indeed proved that $\eta_*$ concentrates on $\mc M$, that is $\eta_*(\mc M) = 1$.
\end{proof}

As a consequence of the previous theorem, we obtain the following convergence when the objective function $f(\cdot)$ has multiple non-trivial minimizers. It indicates that when $f(\cdot)$ admits multiple global minimizers, the consensus point $\mathfrak{m}_{\alpha}(\rho^{\alpha}_{t})$ (or the rescaled mean $\kappa^{-1}\,\EE\left[\,\OX_t\,\right]$) of the CBO dynamics converges to a point in $\mathcal{M}$, that is to one of the global minimizers.

\begin{corollary}\label{cor: conv multi}
In the situation of Theorem \ref{thm:conver3}, the following hold
\begin{equation*}
\begin{aligned}
    & \lim\limits_{\alpha\to \infty}\,\lim\limits_{t\to \infty}  \operatorname{dist}\!\big(\mathfrak{m}_{\alpha}(\rho^{\alpha}_{t}), \,\mathcal{M}\big) 
    =\lim\limits_{\alpha\to \infty} \operatorname{dist}\!\big(\mathfrak{m}_{\alpha}(\rho^{\alpha}_{*}), \,\mathcal{M}\big) 
    = 0,\\
    \text{and} \quad  & \lim\limits_{\alpha\to \infty} \, \lim\limits_{t\to\infty} \operatorname{dist}\!\left( \kappa^{-1}\,\EE\left[\,\OX_t\,\right],\,\mathcal{M}\right) = 0. 
\end{aligned}
\end{equation*}
\end{corollary}

\begin{proof}
Let $\OX_.$ be the solution to the rescaled CBO \eqref{CBO kappa}, and let $\rho_*^\alpha$ be its unique invariant probability measure whose existence and uniqueness are given by Theorem \ref{thm: existence} and Theorem \ref{thm:long time} (see also Remark \ref{rmk: configuration 3} for a discussion on the parameters). Besides Assumption \ref{assum1}, we are assuming $f(\cdot)$ to satisfy Assumption \ref{assum4} as well, and recall $\mathcal{M}$ being the set of its minimizers. 
We proceed in several steps.

\textit{Step 1. (We show $\lim\limits_{\alpha\to \infty} \operatorname{dist}(\mathfrak{m}_{\alpha}(\rho^{\alpha}_{*}), \,\mathcal{M}) = 0$.)}\\
Recall 
\begin{equation*}
    \eta_*^\alpha(A)
    = \frac{\int_{A}e^{-\alpha f(x)}\rho_*^\alpha(\dd x)}{\int_{\RR^d}e^{-\alpha f(y)}\rho_*^\alpha(\dd y)},\quad A\in \mc{B}(\RR^d) 
    \quad \text{ and } \quad
    \mathfrak{m}_{\alpha}(\rho^{\alpha}_{*}) = \int_{\mathbb{R}^{d}}\,x\; \eta_*^\alpha(\dd x).
\end{equation*}
We have
\begin{equation}\label{eq:dist proof M}
\begin{aligned}
    \operatorname{dist}(\mathfrak{m}_{\alpha}(\rho^{\alpha}_{*}), \,\mathcal{M}) 
    & = \inf\limits_{z\in\mathcal{M}}|\mathfrak{m}_{\alpha}(\rho^{\alpha}_{*}) - z| \\
    & = \inf\limits_{z\in\mathcal{M}}\left|\int_{\mathbb{R^{d}}}x\; \eta_*^\alpha(\dd x)  - z\right| 
    \leq \inf\limits_{z\in\mathcal{M}} \int_{\mathbb{R}^{d}} |x-z|\; \eta_*^\alpha(\dd x).
\end{aligned}
\end{equation}
Now we prove that 
\begin{equation}\label{exchange ineq}
    \inf\limits_{z\in\mathcal{M}} \int_{\mathbb{R}^{d}} |x-z|\; \eta_*^\alpha(\dd x) \leq \int_{\mathbb{R}^{d}} \operatorname{dist}(x,\mathcal{M})\, \eta_*^\alpha(\dd x).
\end{equation}
Let $\varepsilon>0$, and let $z_{\varepsilon}$ be a quasi-minimizer for $\min\limits_{z\in\mathcal{M}} |x-z| = \operatorname{dist}(x,\mathcal{M})$, that is
\begin{equation*}
    |x-z_{\varepsilon}| \leq \varepsilon + \operatorname{dist}(x,\mathcal{M}).
\end{equation*}
Then one gets
\begin{equation*}
\begin{aligned}
    \int_{\mathbb{R}^{d}} |x-z_{\varepsilon}| \, \eta_*^\alpha(\dd x)\leq \varepsilon + \int_{\mathbb{R}^{d}}\operatorname{dist}(x,\mathcal{M})\,\eta_*^\alpha(\dd x).
\end{aligned}
\end{equation*}
Moreover, the following always holds
\begin{equation*}
    \inf\limits_{z\in \mathcal{M}} \int_{\mathbb{R}^{d}} |x-z| \, \eta_*^\alpha(\dd x) 
    \leq \int_{\mathbb{R}^{d}} |x-z_{\varepsilon}| \, \eta_*^\alpha(\dd x).
\end{equation*}
Combining the latter two inequalities, one obtains
\begin{equation*}
    \inf\limits_{z\in \mathcal{M}} \int_{\mathbb{R}^{d}} |x-z| \, \eta_*^\alpha(\dd x) 
    \leq 
    \varepsilon + \int_{\mathbb{R}^{d}}\operatorname{dist}(x,\mathcal{M})\,\eta_*^\alpha(\dd x).
\end{equation*}
These quantities being independent from $\varepsilon$, we can send $\varepsilon\to 0$ and get the desired result \eqref{exchange ineq}. 

Next, we continue with \eqref{eq:dist proof M} which together with \eqref{exchange ineq} yield
\begin{equation*}
\begin{aligned}
    \operatorname{dist}(\mathfrak{m}_{\alpha}(\rho^{\alpha}_{*}), \,\mathcal{M})
    & \leq \int_{\mathbb{R}^{d}}\operatorname{dist}(x,\mathcal{M})\,\eta_*^\alpha(\dd x)\\
    & = \int_{\mathcal{M}_{\varepsilon}}     \operatorname{dist}(x,\mathcal{M})\,\eta_*^\alpha(\dd x) + \int_{\mathcal{M}_{\varepsilon}^{c}} \operatorname{dist}(x,\mathcal{M})\,\eta_*^\alpha(\dd x)
\end{aligned}
\end{equation*}
where $\varepsilon>0$ is arbitrarily fixed, and $\mc{M}_\varepsilon$ is as defined in Remark \ref{rmk:flat isolated}. 

The integrand in the first integral is upper-bounded by $\varepsilon$ since $x\in \mc{M}_{\varepsilon}$, which means $\operatorname{dist}(x,\mathcal{M}) \leq \varepsilon$. Thus we have $\int_{\mc{M}_{\varepsilon}} \operatorname{dist}(x,\mathcal{M})\, \eta_*^\alpha(\dd x) \leq \varepsilon$. 

For the second integral, we shall proceed as in the proof of the second statement in Corollary \ref{cor: conv to min}.  
Using Cauchy-Schwarz inequality, we have
\begin{equation*}
\begin{aligned}
    \int_{\mc{M}_{\varepsilon}^{c}} \operatorname{dist}(x,\mathcal{M})\, \eta_*^\alpha(\dd x)
    & \leq \bigg(\eta_*^\alpha(\mc{M}_{\varepsilon}^{c})\bigg)^{\frac{1}{2}} \bigg(\int_{\mc{M}_{\varepsilon}^{c}} \operatorname{dist}(x,\mathcal{M})^{2}\, \eta_*^\alpha(\dd x)\bigg)^{\frac{1}{2}}\\
    & \leq \bigg(\eta_*^\alpha(\mc{M}_{\varepsilon}^{c})\bigg)^{\frac{1}{2}} \bigg(2\,\int_{\mc{M}_{\varepsilon}^{c}} |x|^{2} + |z|^{2}\, \eta_*^\alpha(\dd x)\bigg)^{\frac{1}{2}}\quad \text{ for some } z\in\mathcal{M} \\
    & \leq \bigg(\eta_*^\alpha(\mc{M}_{\varepsilon}^{c})\bigg)^{\frac{1}{2}} \bigg(2\,|z|^{2} + 2\,\eta_*^\alpha(|\cdot|^{2})\bigg)^{\frac{1}{2}}.
\end{aligned}
\end{equation*}
Therefore, we obtain
\begin{align*}
    \operatorname{dist}(\mathfrak{m}_{\alpha}(\rho^{\alpha}_{*}), \,\mathcal{M}) \leq \varepsilon + \bigg(\eta_*^\alpha(\mc{M}_{\varepsilon}^{c})\bigg)^{\frac{1}{2}} \bigg(2\,|z|^{2} + 2\,\eta_*^\alpha(|\cdot|^{2})\bigg)^{\frac{1}{2}}.
\end{align*}
Recalling \eqref{lemeq2} and \eqref{eq:sec}, the second moment $\eta_*^\alpha(|\cdot|^{2})$ is bounded as in \eqref{bound second moment eta}.  Finally, using Theorem \ref{thm:conver3} (see in particular \eqref{lim thm}), we know that $\eta_*^\alpha(\mc{M}_{\varepsilon}^{c}) \,\to 0\,$ as $\alpha \to \infty$. Here $\varepsilon$ being arbitrarily chosen, this yields $\lim\limits_{\alpha\to \infty}\operatorname{dist}(\mathfrak{m}_{\alpha}(\rho^{\alpha}_{*}), \,\mathcal{M}) =0$.

\textit{Step 2. (We show $\lim\limits_{\alpha\to \infty}\,\lim\limits_{t\to \infty}  \operatorname{dist}(\mathfrak{m}_{\alpha}(\rho^{\alpha}_{t}), \,\mathcal{M}) =0 $.)}\\
We shall use Corollary \ref{cor: long time of m}, together with the result in \textit{Step 1}. Indeed, using the triangular inequality, one has
\begin{equation*}
    \operatorname{dist}(\mathfrak{m}_{\alpha}(\rho^{\alpha}_{t}), \,\mathcal{M}) \leq \operatorname{dist}(\mathfrak{m}_{\alpha}(\rho^{\alpha}_{*}), \,\mathcal{M}) + |\mathfrak{m}_{\alpha}(\rho^{\alpha}_{t}) - \mathfrak{m}_{\alpha}(\rho^{\alpha}_{*})|.
\end{equation*}
It suffices now to take the limit first as $t\to \infty$ using Corollary \ref{cor: long time of m}, then send $\alpha\to \infty$ using the conclusion of \textit{Step 1} above.

\textit{Step 3. (We show $\lim\limits_{\alpha\to \infty} \, \lim\limits_{t\to\infty} \operatorname{dist}\left( \kappa^{-1}\,\EE\left[\,\OX_t\,\right],\,\mathcal{M}\right) = 0$.)}\\
Using the triangular inequality, one has
\begin{equation*}
    \operatorname{dist}\left( \kappa^{-1}\,\EE\left[\,\OX_t\,\right],\,\mathcal{M}\right) \leq \operatorname{dist}(\mathfrak{m}_{\alpha}(\rho^{\alpha}_{*}), \,\mathcal{M}) + |\kappa^{-1}\,\EE\left[\,\OX_t\,\right] - \mathfrak{m}_{\alpha}(\rho^{\alpha}_{*})|.
\end{equation*}
The conclusion thus follows using the first statement of Corollary \ref{cor: conv to min} which holds under Assumption \ref{assum1}, together with \textit{Step 1} above. 
\end{proof}

\begin{remark}
A result more general than \eqref{exchange ineq} can be found in \cite[Proposition 3.1]{kouhkouh2024viscous}.
\end{remark}

To validate Corollary~\ref{cor: conv multi}, we apply our CBO dynamics to the two-dimensional objective function
\[
f(x,y):= \big((x-1)^2 + (y-1)^2\big)\big((x+1)^2 + (y+1)^2\big),
\]
which possesses two global minimizers at $(1,1)$ and $(-1,-1)$. In Figure~\ref{fig:multi}, we plot the trajectory of the consensus points $\mathfrak{m}_{\alpha}(\rho^{N}_{t})$ generated by the CBO algorithm. The results show that, depending on the initial particle distribution, the method converges to one of the two global minimizers, thereby confirming Corollary~\ref{cor: conv multi}. The simulation parameters are chosen as $N = 1000$, $T = 4$, $\mathrm{d}t = 0.01$, $\lambda = 1$, $\sigma = 0.3$, $\kappa = 0.9$, and $\delta = 0.8$.

\begin{figure}[H]%[htb!]
\centering
\begin{minipage}{0.49\textwidth}
\centering
\includegraphics[width=\linewidth,height=6cm]{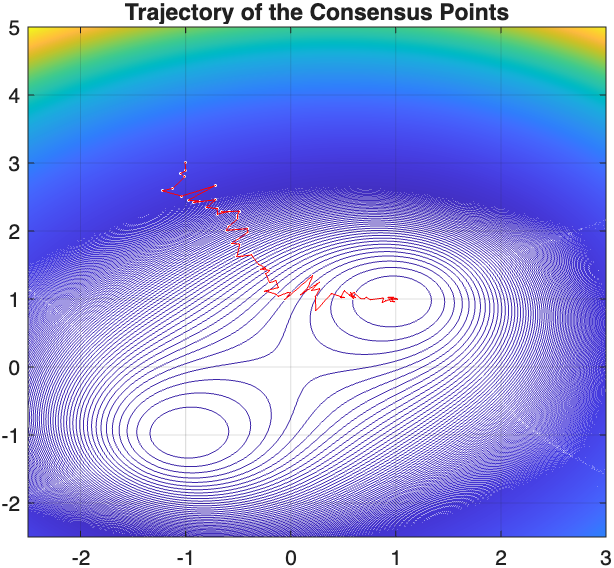}
\end{minipage}
\hfill
\begin{minipage}{0.49\textwidth}
\centering
\includegraphics[width=\linewidth,height=6cm]{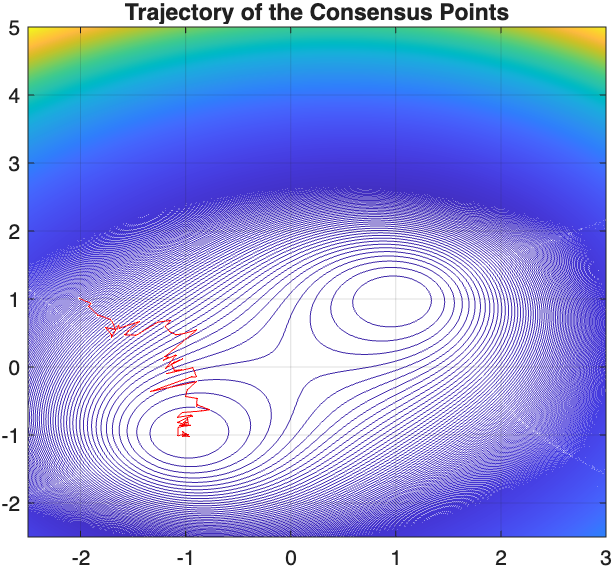}
\end{minipage}
\caption{%
Numerical test on the objective function $f(x,y)$, which has two global minimizers at $(1,1)$ and $(-1,-1)$.\\ 
\textbf{Left:} The particles are initialized uniformly in $[-2,-1]\times [3,4]$, a region distant from the global minimizers. The CBO method successfully converges to the global minimizer $(1,1)$. \\
\textbf{Right:} The particles are initialized uniformly in $[-3,-2]\times [1,2]$. The CBO method successfully converges to the global minimizer $(-1,-1)$.
}
\label{fig:multi}
\end{figure}

\section{Numerical aspects}\label{sec: num}

In this section, we investigate numerical aspects of the rescaled CBO, with particular emphasis on the sensitivity of the algorithm to the parameters $\kappa$ and $\delta$. 

Our results indicate that the rescaled CBO offers significant advantages over the standard CBO, especially when the initial particle distribution does not contain the global minimizer. 

To this end, we first apply the rescaled CBO dynamics using the standard Euler–Maruyama scheme to the two-dimensional Rastrigin function $R(\cdot,\cdot)$, which features a unique global minimizer at $x_* = (1, 1)^{\top}$ 
\begin{equation*}
    R(x_1,x_2) = 20 + \left(x_1^2 - 10\cos(2\pi x_1)\right) + \left(x_2^2 - 10\cos(2\pi x_2)\right).
\end{equation*}
The particles are initially uniformly distributed in the square $[2, 3] \times [2, 3]$, situated away from $x_*$ to test the algorithm's global search capabilities. We fix the parameters at $\lambda = 1, \sigma = 0.5, \alpha = 10^{15}, T = 100,$ and $\Delta t = 0.01$ to evaluate the sensitivity of the success rate across various choices of $\kappa$ and $\delta$. In Table \ref{table:hh1}, the number of particles is $N=100$. In Table~\ref{table:hh2}, it is $N=500$.    A test is deemed successful if  $|\mathfrak{m}_{\alpha}(\rho_T^N)-x_*|<0.05$; otherwise, it fails. The success rate is calculated based on $100$ test runs.

\begin{table}[H]
    \centering
    \footnotesize % Reduces font size for both tables
    
    % --- FIRST TABLE ---
    \begin{minipage}{0.48\textwidth}
        \centering
        \resizebox{\textwidth}{!}{% Resize to fit the minipage width
        \begin{tabular}{lllllllllll}
        \hline
        \diagbox{$\delta$}{Rate}{$\kappa$} 
               & 0.1   & 0.2  & 0.3  & 0.4  & 0.5  & 0.6  & 0.7  & 0.8  & 0.9  & 1.0  \\ \hline
             0 & 0     & 0    & 0    & 0    & 0    & 0    & 0    & 0    & 0    & 0       \\
           0.5 & 0     & 0    & 0    & 0    & 0    & 0    & 0.03 & 0.98 & 1.00 & 1.00  \\
           1.0 & 0.02  & 0.03 & 0.05 & 0.08 & 0.08  & 0.19 & 0.49 & 0.87 & 0.96   & 0.98  \\
           1.5 & 0.05  & 0.06 & 0.13 & 0.16 & 0.31 & 0.38 & 0.52 & 0.72 & 0.84   & 0.79  \\
           2.0 & 0.19  & 0.18 & 0.21 & 0.26 & 0.30 & 0.33 & 0.39 & 0.61 & 0.68   & 0.60  \\
           2.5 & 0.16  & 0.21 & 0.15 & 0.29 & 0.27 & 0.40 & 0.43 & 0.54 & 0.51   & 0.49 \\
           3.0 & 0.14  & 0.17 & 0.17 & 0.15 & 0.25 & 0.22 & 0.37 & 0.44 & 0.42   & 0.42 \\
           3.5 & 0.21  & 0.26 & 0.20 & 0.27 & 0.30 & 0.37 & 0.25 & 0.25 & 0.35   & 0.33  \\
           4.0 & 0.15  & 0.19 & 0.26 & 0.20 & 0.24 & 0.22 & 0.28 & 0.36 & 0.32   & 0.32  \\
           4.5 & 0.16  & 0.15 & 0.18 & 0.16 & 0.31 & 0.24 & 0.27 & 0.23 & 0.34   & 0.22 \\
           5.0 & 0.11  & 0.21 & 0.20 & 0.18 & 0.21 & 0.14 & 0.28 & 0.26 & 0.31   & 0.33 \\ \hline
        \end{tabular}}
        \caption{Success rate for $N=100$.}
        \label{table:hh1}
    \end{minipage}
    \hfill % Adds spacing between the two minipages
    % --- SECOND TABLE ---
    \begin{minipage}{0.48\textwidth}
        \centering
        \resizebox{\textwidth}{!}{%
        \begin{tabular}{lllllllllll}
        \hline
        \diagbox{$\delta$}{Rate}{$\kappa$} 
               & 0.1   & 0.2  & 0.3  & 0.4  & 0.5  & 0.6  & 0.7  & 0.8  & 0.9  & 1.0  \\ \hline
             0 & 0     & 0    & 0    & 0    & 0    & 0    & 0    & 0    & 0    & 0       \\
           0.5 & 0     & 0.01 & 0    & 0.03 & 0.02 & 0.04 & 1.00 & 1.00 & 1.00 & 1.00  \\
           1.0 & 0.08  & 0.13 & 0.15 & 0.31 & 0.61 & 0.87 & 0.97 & 1.00 & 1.00   & 1.00  \\
           1.5 & 0.30  & 0.29 & 0.40 & 0.56 & 0.76 & 0.91 & 0.94 & 0.99 & 0.99   & 1.00  \\
           2.0 & 0.46  & 0.37 & 0.50 & 0.60 & 0.72 & 0.81 & 0.92 & 0.93 & 0.99   & 0.98  \\
           2.5 & 0.40  & 0.37 & 0.50 & 0.61 & 0.66 & 0.82 & 0.84 & 0.90 & 0.94   & 0.91 \\
           3.0 & 0.47  & 0.39 & 0.51 & 0.63 & 0.67 & 0.75 & 0.74 & 0.83 & 0.81   & 0.81 \\
           3.5 & 0.49  & 0.48 & 0.51 & 0.44 & 0.53 & 0.64 & 0.67 & 0.67 & 0.76   & 0.72  \\
           4.0 & 0.40  & 0.43 & 0.49 & 0.59 & 0.48 & 0.62 & 0.65 & 0.57 & 0.72   & 0.71  \\
           4.5 & 0.40  & 0.43 & 0.40 & 0.52 & 0.46 & 0.50 & 0.58 & 0.53 & 0.55   & 0.62 \\
           5.0 & 0.34  & 0.51 & 0.47 & 0.37 & 0.48 & 0.48 & 0.54 & 0.43 & 0.55   & 0.51 \\ \hline
        \end{tabular}}
        \caption{Success rate for $N=500$.}
        \label{table:hh2}
    \end{minipage}
\end{table}

As observed in the table, optimal performance is generally achieved by choosing a large $\kappa$ and a small $\delta$ (the upper-right region of the table). Notably, the standard CBO configuration ($\kappa=1, \delta=0$) fails completely with a success rate of $0$. These results further validate the motivation for introducing the scaling parameter $\kappa \in (0, 1)$ and the non-degenerate diffusion $\delta > 0$ to enhance convergence. 
Furthermore, one can observe in Table~\ref{table:hh2} that increasing the number of particles generally improves the success rate.

We further apply the CBO dynamics to the Ackley function
\[
A(x) := -20 \exp\!\left(-0.2 |x - x_*|\right) - \exp\!\left( \frac{1}{d} \sum_{i=1}^d \cos(2\pi (x_i - x_{*,i})) \right) + e + 20,
\]
in dimension $d = 20$, where the unique global minimizer is $x_* = (1,\dots,1)^\top$.

\begin{figure}[H]%[htb!]
\centering
\begin{minipage}{0.49\textwidth}
\centering
\includegraphics[width=\linewidth]{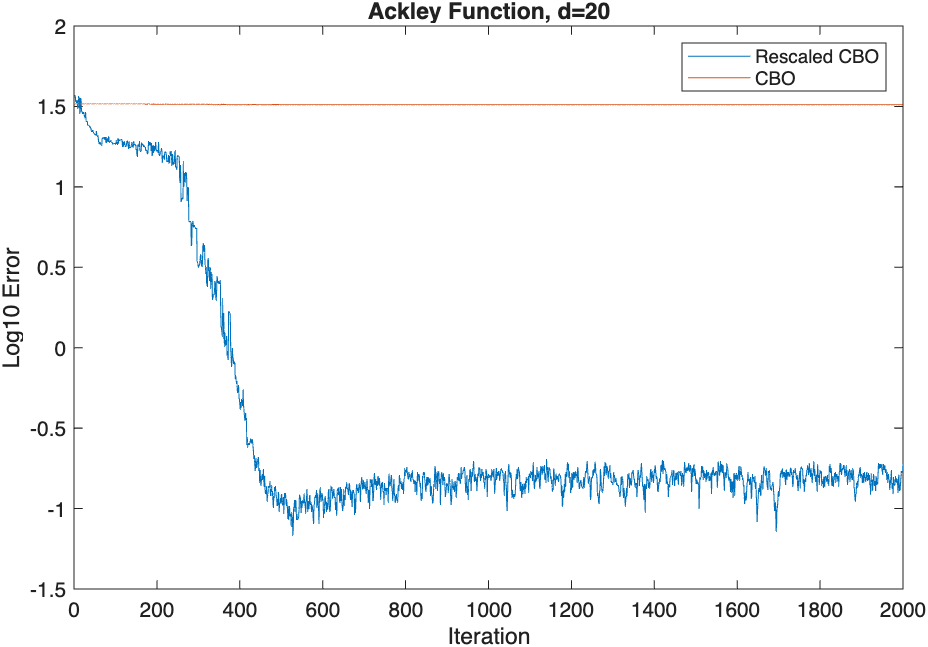}
\end{minipage}
\hfill
\begin{minipage}{0.49\textwidth}
\centering
\includegraphics[width=\linewidth]{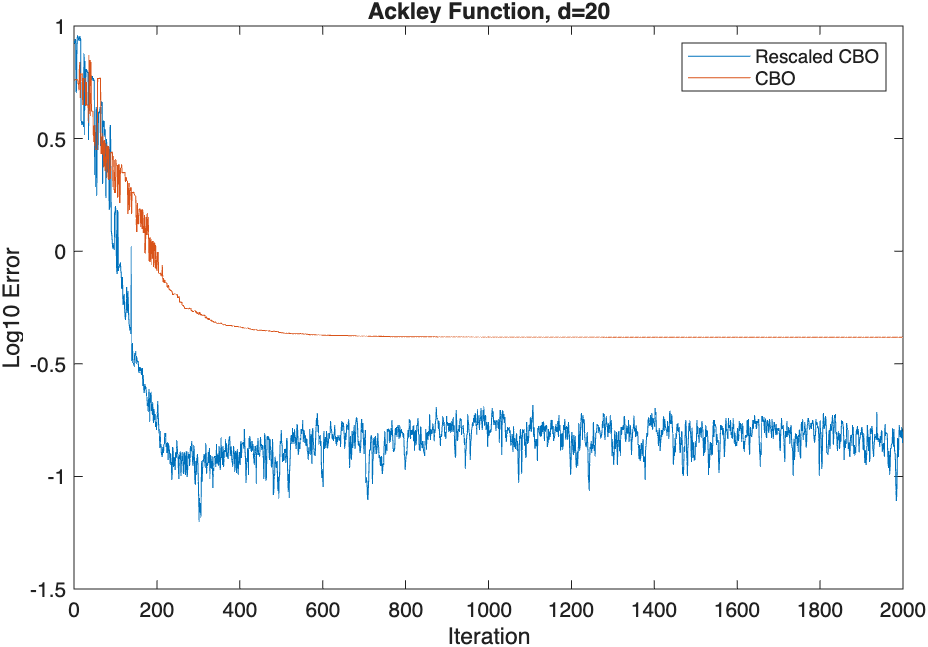}
\end{minipage}
\caption{%
We test the Ackley function in dimension $d = 20$.\\ 
\textbf{Left:} The particles are initialized uniformly in $[2,3]^{20}$ (which \textit{excludes} $x_*$). The \textit{rescaled} CBO ($\kappa = 0.9$, $\delta = 0.8$) converges to the global minimizer, while the \textit{standard} CBO stalls at a local minimum. \\
\textbf{Right:} The particles are drawn from $\mathcal{N}((2,\dots,2)^\top, \mathds{I}_{20})$ (whose support \textit{includes} $x_*$). The \textit{rescaled} CBO (with $\kappa = 0.9$ and $\delta = 0.8$) still outperforms the \textit{standard} CBO. 
}
\label{fig:ackley_comparison}
\end{figure}
In Figure \ref{fig:ackley_comparison}, we fix the parameters $\lambda = 1$, $\sigma = 0.3$, $\alpha = 10^{15}$, $N = 500\,000$, $T = 20$, and $\Delta t = 0.01$.
In Figure~\ref{fig:ackley_comparison}-(Left), the particles are initially sampled uniformly from the hypercube $[2,3]^{20}$, which does not contain $x_*$. The plot shows the evolution of the error $\log(|\mathfrak{m}_{\alpha}(\rho_t^N) - x_*|)$ over time $t$. One observes that the standard CBO becomes trapped in a local minimizer, whereas the rescaled CBO with $\kappa = 0.9$ and $\delta = 1$ successfully converges to the global minimizer. In Figure~\ref{fig:ackley_comparison}-(Right), we consider an initial distribution $\mathcal{N}((2,\dots,2)^\top, \mathds{I}_{20})$, whose support includes $x_*$. Even in this more favorable setting, the rescaled CBO demonstrates better performance, highlighting its robustness and effectiveness beyond merely escaping poor initialization.

Several additional numerical experiments (using dimensions up to 20 and a variety of objective functions) are reported in \cite[Section 5]{herty2025multiscale}. In fact, the optimization problem studied in \cite{herty2025multiscale} is multi-level such as
\begin{equation*}
    \begin{aligned}
        &\min\limits_{x\in\mathbb{R}^n} \, F(x,y)\\
        & \quad\text{s.t.}\quad  y\in \argmin_{y\in\mathbb{R}^m} \, G(x,y),
    \end{aligned}
\end{equation*}
where $F,G:\mathbb{R}^{n}\times \mathbb{R}^{m}\to \mathbb{R}$ for some $n,m\in \mathbb{N}$.
In order to fit it into the present framework, it suffices to take the objective function $F$  to be constant (a simplification) and drop the dependence on its corresponding $x$ variable (by setting it to a fixed constant) in the objective function $G$. The resulting optimization problem becomes
\begin{equation*}
    \left.
    \begin{aligned}
        &\min\limits_{x\in\mathbb{R}^n} \quad 1\\
        & \quad\text{s.t.}\quad  y\in \argmin_{y\in\mathbb{R}^m} \, G(0,y)
    \end{aligned}
    \quad\right \}
    \quad  \Longleftrightarrow \qquad \min\limits_{y\in\mathbb{R}^m} \, \overline{G}(y)
\end{equation*}
where $\overline{G}(y):= G(0,y)$. Hence, the multi-scale dynamics \cite[equation (3.1)]{herty2025multiscale} reduces to $(0, \mathds{Y}_{t})$ where the process $\mathds{Y}$ solves  \eqref{CBO kappa}. 
With this simplification, the model in \cite{herty2025multiscale} encompasses the one considered here, and the remainder of  \cite[Section 5]{herty2025multiscale} is fully compatible with our setting. Consequently, the numerical tests and the results obtained therein apply directly and would be identical in our case. See also \cite[Remark 5.1]{herty2025multiscale}. 

\section{Conclusion}\label{sec: conclusion}

\subsection{Summary of main findings}

We have studied asymptotic properties of the rescaled CBO dynamics  
\begin{equation*}
    \rd \OX_t =-\lambda\left(\OX_t - \kappa\,\mathfrak{m}_{\alpha}(\rho_t^\alpha)\right)\,\dt + \sigma\left(\delta\,\mathds{I}_{d} + D(\OX_t - \kappa\, \mathfrak{m}_{\alpha}(\rho_t^\alpha)) \right)\rd B_t,\quad \rho_t^\alpha =\text{Law}(\OX_t)
\end{equation*}
under suitable conditions on the parameters, for example those in Remark \ref{rmk: configuration 3}.  
We do not make any assumptions on the initial distribution $\OX_{0}\sim\rho_0$ besides satisfying $ \rho_0\in\mathscr{P}_{4}(\RR^d)$. Supposing that $f(\cdot)$ satisfies Assumption \ref{assum1}:
\begin{itemize}
    \item We proved that such a process admits a unique invariant probability measure $\rho^{\alpha}_{*}$, and its law enjoys a $\mathbb{W}_{2}$--exponential contraction \eqref{contraction}.
    \item When $f(\cdot)$ has a unique minimizer $x_*$, we proved that 
    \begin{equation*}
    \begin{aligned}
        & \lim_{\alpha\to \infty}\lim_{t\to \infty}\; \mathfrak{m}_{\alpha}(\rho^{\alpha}_{t}) = \lim_{\alpha\to \infty}\; \mathfrak{m}_{\alpha}(\rho^{\alpha}_{*}) = x_{*},\\
        \text{ and } \; & \lim_{\alpha\to \infty}\lim_{t\to \infty} \; \kappa^{-1} \EE\left[\,\OX_t\,\right]=\lim_{\alpha\to \infty}\; \kappa^{-1} \int_{\RR^d}x \, \dd\rho^{\alpha}_{*} (x) \, = \,  x_{*},\\
    \end{aligned}
    \end{equation*}
    
    \item 
    When $f(\cdot)$ has multiple minimizers as $\mathcal{M}$, we proved that
    \begin{equation*}
    \begin{aligned}
        & \lim\limits_{\alpha\to \infty}\,\lim\limits_{t\to \infty}  \operatorname{dist}(\mathfrak{m}_{\alpha}(\rho^{\alpha}_{t}), \,\mathcal{M}) 
        =\lim\limits_{\alpha\to \infty} \operatorname{dist}(\mathfrak{m}_{\alpha}(\rho^{\alpha}_{*}), \,\mathcal{M}) 
        = 0,\\
        \text{and} \quad  & \lim\limits_{\alpha\to \infty} \, \lim\limits_{t\to\infty} \operatorname{dist}\!\left( \kappa^{-1}\,\EE\left[\,\OX_t\,\right],\,\mathcal{M}\right) = 0. 
    \end{aligned}
    \end{equation*}
\end{itemize}
The absence of any conditions on the initial distribution is precisely what ensures the \textit{global} nature of the convergence, making this result, to the best of our knowledge, the first --and, as of now, the only-- one that guarantees global convergence as both time and inverse temperature $\alpha$ tend to infinity.

\subsection{Future research directions}

The present analysis paves the way for several extensions. In particular, we intend to investigate the following topics in future work:
\begin{itemize}
    \item A faithful convergence for the polarized CBO \cite{bungert2022polarized,fornasier2025pde} to effectively find multiple minimizers (not only one of them as it is proved in the present paper). 
    \item An adaptive convergence where the growth of $\alpha$ is a function of time.
    \item An extension to stochastic optimization as in \cite{bonandin2025consensus,bellavia2025discrete}, where the objective function depends on a random parameter.
    \item A sharper qualitative and quantitative analysis on the values of the parameters $\kappa, \delta, N, \alpha, \sigma, T$. Some insights have been discussed in \cite[Section 5]{herty2025multiscale}. 
\end{itemize}

\section*{Acknowledgments}

The authors wish to thank the three referees for their valuable comments which helped improve the paper.

\appendix

\section{Existence of invariant measure}\label{app: existence}

This section is devoted to verifying Theorem \ref{thm: existence} on the existence of an invariant measure to the distribution-dependent SDE (DDSDE) \eqref{CBO kappa} which we recall is
\begin{equation}\label{CBO appendix}
\rd X_t=-\lambda(X_t-\kappa\, \mathfrak{m}_{\alpha}(\operatorname{Law}(X_t))\dt+\sigma\left(\delta\,\mathds{I}_{d} + D(X_t-\kappa\,\mathfrak{m}_{\alpha}(\operatorname{Law}(X_t))\right) \rd B_t\,.
\end{equation}
We use the following notation for brevity: 
for $(x,\mu)\in \mathbb{R}^{d}\times \mathscr{P}(\mathbb{R}^{d})$
\begin{equation*}
	b(x,\mu):=-\lambda(x-\kappa\,\mathfrak{m}_{\alpha}(\mu)),\quad \text{ and }\quad \bm{\sigma}(x,\mu):=\sigma\delta\,\mathds{I}_{d} + \sigma D(x-\kappa\,\mathfrak{m}_{\alpha}(\mu)),
\end{equation*}
where $\lambda, \sigma,\kappa,\alpha$ are positive constants.

We shall provide  a self-contained verification of the existence of an invariant measure to \eqref{CBO appendix} which we recall was proven in \cite[Proposition 3.4]{huang2024self} (see also Appendix A.2 therein). 
To do so, we shall verify the validity of the assumptions needed for the result of S.-Q. Zhang \cite[Theorem 2.2]{zhang2023existence} which will guarantee existence of an invariant measure in $\mathscr{P}_{2,R}(\mathbb{R}^{d})$ for some sufficiently large fixed $R>0$.

\textbullet\; Assumption \cite[(H1)]{zhang2023existence}\\
We verify the validity of this assumption with $r_{1} = r_{2} = 1$ and $r_{3} = 1+r_{2}= 2$. Then we have, for all $\nu \in \mathscr{P}_{2}(\mathbb{R}^{d})$,
\begingroup
\begin{align*}\allowdisplaybreaks
    & 2 \langle b(x,\nu),x\rangle + \|\bm\sigma(x,\nu)\|^{2}  = 2\langle -\lambda x + \lambda\,\kappa \, \mathfrak{m}_{\alpha}(\nu), x \rangle + \text{trace}(\bm{\sigma}\bm{\sigma}(x,\mu)^{\top})\\
    & \quad \quad = -2 \lambda |x|^{2} + 2\lambda \,\kappa\, \langle x, \mathfrak{m}_{\alpha}(\nu) \rangle + \sum\limits_{i=1}^{d}\left( \sigma\,\delta + |\{x-\kappa\, \mathfrak{m}_{\alpha}(\nu)\}_{i}| \right)^{2}\\
    & \quad \quad \leq -2 \lambda |x|^{2} + \lambda\,\kappa \left(|x|^{2} + |\mathfrak{m}_{\alpha}(\nu)|^{2}\right) + 2\,d\,\left(\sigma\,\delta\right)^{2} + 2 |x-\kappa\,\mathfrak{m}_{\alpha}(\nu)|^{2}\\
    & \quad \quad \leq -2 \lambda |x|^{2} + \lambda\,\kappa \left(|x|^{2} + |\mathfrak{m}_{\alpha}(\nu)|^{2}\right) + 2\,d\,\left(\sigma\,\delta\right)^{2} + 4 |x|^{2} + 4\,\kappa^{2}\,|\mathfrak{m}_{\alpha}(\nu)|^{2}\\
    & \quad \quad \leq (-2\lambda + \lambda\,\kappa + 4)|x|^{2} + (\lambda\,\kappa + 4\,\kappa^{2})|\mathfrak{m}_{\alpha}(\nu)|^{2} + 2\,d\,(\sigma\,\delta)^{2}\\
    & \quad \quad \leq -(2\lambda - \lambda\,\kappa - 4)|x|^{2} + (\lambda\,\kappa + 4\,\kappa^{2})|\mathfrak{m}_{\alpha}(\nu)|^{2} + 2\,d\,(\sigma\,\delta)^{2}
\end{align*}
\endgroup
Using Lemma \ref{lem: useful estimates}, we can bound $|\mathfrak{m}_{\alpha}(\nu)|^{2}$ as in \eqref{lemeq2} and obtain
\begin{equation}\label{H1}
\begin{aligned}
    & 2 \langle b(x,\nu),x\rangle + \|\bm\sigma(x,\nu)\|^{2} \\
    & \quad  \leq -(2\lambda - \lambda\,\kappa - 4)|x|^{2} + (\lambda\,\kappa + 4\,\kappa^{2})[b_1 + b_2 \nu(|\cdot|^{2})] + 2\,d\,(\sigma\,\delta)^{2}\\
    & \quad  \leq -(2\lambda - \lambda\,\kappa - 4)|x|^{2} + (\lambda\,\kappa+4\,\kappa^{2})\,b_{2}\,\nu(|\cdot|^{2}) + 2\,d\,(\sigma\,\delta)^{2} + b_{1}\,(\lambda\,\kappa+4\,\kappa^{2})\\
    & \quad  = -\widetilde{C}_{1}\, |x|^{2} + \widetilde{C}_{2} + \widetilde{C}_{3}\,\nu(|\cdot|^{2})
\end{aligned}
\end{equation}
where
\begin{equation*}
\begin{aligned}
    \widetilde{C}_{1}  = 2\lambda - \lambda\,\kappa - 4, \qquad
    \widetilde{C}_{2}  = 2\,d\,(\sigma\,\delta)^{2} + b_{1}\,(\lambda\,\kappa+4\,\kappa^{2}), \qquad 
    \widetilde{C}_{3}  = (\lambda\,\kappa+4\,\kappa^{2})\,b_{2}.
\end{aligned}
\end{equation*}
We need to verify that $\widetilde{C}_{1}>0$ and $\widetilde{C}_{1} > \widetilde{C}_{3}$. Observe that $\widetilde{C}_{3}  = \kappa\,(\lambda+4\,\kappa)\,b_{2} = \mathcal{O}(\kappa)$ which means that it can be made arbitrarily small when $0<\kappa\ll 1$. We have 
\begin{equation}\label{equiv C1 C3}
\begin{aligned}
    \widetilde{C}_{1} > \widetilde{C}_{3} \; \Leftrightarrow\; 2\lambda - \lambda\,\kappa - 4 
     > (\lambda\,\kappa+4\,\kappa^{2})\,b_{2}
    \; \Leftrightarrow\; 2\lambda - 4  > \lambda\,\kappa\,(1+b_2) + 4\,\kappa^{2} b_2
\end{aligned}
\end{equation}
Then, \textbf{sufficient} conditions are for example $\kappa < \frac{1}{2}(1+b_2)^{-1} \; \text{ and } \; \lambda >4$, where $b_2$ is as defined in \eqref{constant b1 b2}. Indeed in this case, recalling $\kappa\in (0,1)$, one has
\begin{equation*}
\begin{aligned}
    & \left(\lambda\,\kappa (1+b_2) < \frac{\lambda}{2}\right) \;\; 
    \text{ and } \;\;  \bigg(\kappa\,b_2 < 1  \Rightarrow\; 4\,\kappa^{2}\,b_2 < 2\bigg)
    \; \Rightarrow\;  \lambda\,\kappa (1+b_2) + 4\kappa^{2}\, b_2 < \frac{\lambda}{2} + 2
\end{aligned}
\end{equation*}
Finally, $\lambda>4 \Leftrightarrow 2\lambda - 4 > \frac{\lambda}{2} + 2$ together with the previous inequality guarantee \eqref{equiv C1 C3}. Moreover, $\lambda>4$ together with $\kappa\in (0,1)$ ensure $\widetilde{C}_1 >0$.

\textbullet\; Assumption \cite[(H2.i)]{zhang2023existence}\\
For $\nu$ fixed in $\mathscr{P}_{2}(\mathbb{R}^{d})$, we need to check that the drift and diffusion terms are locally Lipschitz, that is: \\
For every $n\in \mathbb{N}$ and $\nu \in \mathscr{P}_{2}(\mathbb{R}^{d})$, there exists $K_{n}>0$ such that for all $|x|\vee |y|\leq n$ we have
\begin{equation*}
    |b(x,\nu) - b(y,\nu)| + \|\bm\sigma(x,\nu) - \bm\sigma(y,\nu)\| \leq K_{n}|x-y|.
\end{equation*}
This is easily verified from the definition of $b$ and $\bm\sigma$ (in fact, $K_{n}=\lambda + \sigma$). 

\textbullet\; Assumption \cite[(H2.ii)]{zhang2023existence}\\
We need to check that the drift has a polynomial growth. More precisely we want to check that: there exists a locally bounded function $\mathfrak{h}:[0,+\infty)\to [0,+\infty)$ such that
\begin{equation*}
    |b(x,\nu)| \leq \mathfrak{h}(\nu(|\cdot|^{2}))\, (1+|x|),\quad x\in \mathbb{R}^{d},\; \nu \in \mathscr{P}_{2}(\mathbb{R}^{d}).
\end{equation*}
This holds true, noting that
%\begin{equation*}
\begingroup
\begin{align*}\allowdisplaybreaks
    |b(x,\nu)| & = \lambda|x- \kappa \,\mathfrak{m}_{\alpha}(\nu)|\leq \lambda\big(|x| + \kappa |\mathfrak{m}_{\alpha}(\mu)|\big)\\
    & \leq \lambda|x| + \lambda\kappa \bigg(b_1 + b_2 \,\nu(|\cdot|^{2})\bigg)^{\frac{1}{2}},\quad \text{using Lemma \ref{lem: useful estimates}}\\
    & \leq \lambda(1+|x|) + \lambda\bigg(b_1 + b_2 \,\nu(|\cdot|^{2})\bigg)^{\frac{1}{2}}(1+|x|), \quad \text{using $\kappa <1$}\\
    & \leq \mathfrak{h}(\nu(|\cdot|^{2}))\, \big(1 + |x| \big)
\end{align*}
\endgroup
%\end{equation*}
where we have set $\mathfrak{h}(\xi) = \lambda \left(1+ (b_1 + b_2 \, \xi)^{\frac{1}{2}} \right)$ for every $\xi\in [0,+\infty)$.

\textbullet\; Assumption \cite[(H3)]{zhang2023existence}\\
We need the drift and diffusion coefficients to be continued on $\mathscr{P}_{2,R}(\mathbb{R}^{d})$ equipped with the Wasserstein metric. This is guaranteed thanks to Lemma \ref{lem: stab}.

\textbullet\; Assumption \cite[(H4)]{zhang2023existence}\\
We need the diffusion matrix to be non-degenerate on $\mathbb{R}^{d}\times \mathscr{P}_{2}(\mathbb{R}^{d})$ , that is
\begin{equation*}
    \bm\sigma(x,\nu)\bm\sigma^{*}(x,\nu)>0,\quad x\in \mathbb{R}^{d},\; \nu \in \mathscr{P}_{2}(\mathbb{R}^{d}).
\end{equation*}
This is guaranteed thanks to the additional term $\delta\,\mathds{I}_{d}$ in the definition of $\bm{\sigma}$, noting that $D(\cdot)$ is a non-negative matrix. Indeed, the matrix $\bm{\sigma}$ is a diagonal matrix whose diagonal entries satisfy $\frac{\sigma^2}{2}(\delta+|(x-\kappa\,\mathfrak{m}_{\alpha}(\mu))_i|)^2 \geq \frac{(\sigma\delta)^{2}}{2}>0$.

\textbullet\; Conclusion:\\
The latter assumptions being satisfied, we can then apply  \cite[Theorem 2.2]{zhang2023existence}, which guarantees that the DDSDE \eqref{CBO appendix} has a stationary distribution. In fact, from the proof in \cite{zhang2023existence} (see the last line in page 8, and  the lines between equation (2.22) and equation (2.23) in page 10),  there exists $R_{0}>0$ depending only on the constants of the problem ($\lambda, \kappa, \sigma, \delta$, together with $b_{1}, b_2$ defined by \eqref{constant b1 b2} in Lemma \ref{lem: useful estimates}) such that the stationary distribution exists in $\mathscr{P}_{2,R}(\mathbb{R}^{d})$, and $R\geq R_{0}$ is determined by the initial distribution. 

The precise conditions that the parameters need to satisfy in order to guarantee existence of an invariant measure are ruled by the inequality $\tilde{C_{1}} > \tilde{C}_{3}$ that is \eqref{equiv C1 C3}. An example of sufficient conditions are $\kappa < \frac{1}{2(1+b_2)} \;\text{ and } \; \lambda >4\;$ where we recall $b_2$ is as defined in \eqref{constant b1 b2}.

\bibliographystyle{amsxport}
\bibliography{bibliography}

\end{document}